\documentclass[reqno,hidelinks,11pt]{amsart}
\usepackage[export]{adjustbox}
\usepackage{amsmath,amsthm,amsfonts,amssymb}
\usepackage{parskip,fullpage}
\usepackage{hyperref}
\usepackage[dvipsnames]{xcolor}
\usepackage{comment}
\usepackage{graphicx,soul}
\usepackage{geometry}
\usepackage[shortlabels]{enumitem}
\usepackage{pgfplots}
\pgfplotsset{compat=1.18}
\usepackage{subcaption}
\hypersetup{
    colorlinks=true,
    citecolor=red
}
\usepackage{float}
\usepackage[
backend=biber,
style=alphabetic,
sorting=nyt,
maxnames=99
]{biblatex}
\numberwithin{equation}{section}
\newtheorem{theorem}{Theorem}[section]
\newtheorem{lemma}[theorem]{Lemma}
\newtheorem{remark}[theorem]{Remark}

\newtheorem{corollary}[theorem]{Corollary}
\newtheorem{proposition}[theorem]{Proposition}

\newtheorem{conjecture}[theorem]{Conjecture}

\newtheorem{problem}[theorem]{Problem}

\def\RR{\mathbb{R}}
\def\CC{\mathbb{C}}

\def\b{\beta}
\def\d{\delta}

\def\rank{\operatorname{rank}}

\def\A{\mathcal{A}}

\def\X{\mathcal{X}}

\def\diag{\operatorname{diag}}

\def\S{\mathcal{S}}
\def\Sym{\operatorname{Sym}}
\def\Re{\operatorname{Re}}
\def\Im{\operatorname{Im}}
\def\tr{\operatorname{tr}}

\newcommand{\suchthat}{\;\ifnum\currentgrouptype=16 \middle\fi|\;}

\title{Eigenvalues of locally positive semidefinite matrices: Non-convexity and Geometry}
\author{Jose Acevedo}
\address{Independent Researcher, Bucaramanga, Colombia}
\email{jacevedo.math@gmail.com}

\author{Grigoriy Blekherman}
\address{School of Mathematics, Georgia Institute of Technology, 686 Cherry Street Atlanta, GA 30332, USA}
\email{greg@math.gatech.edu}

\author{Sebastian Debus}
\address{Fachbereich Mathematik und Statistik, Universität Konstanz, 78457 Konstanz, Germany}
\email{sebastian.debus@uni-konstanz.de}

\author{Seokbin Lee}
\address{School of Mathematics, Georgia Institute of Technology, 686 Cherry Street Atlanta, GA 30332, USA}
\email{slee3379@gatech.edu}

\author{Cordian Riener}
\address{Department of Mathematics and Statistics, UiT - the Arctic University of Norway, 9037 Troms\o, Norway}
\email{cordian.riener@uit.no}
\date{\today}
\subjclass[2020]{Primary 15A99; Secondary 14P10, 90C47}
\keywords{Locally positive semidefinite matrices, spectral geometry, semialgebraic sets, principal minors, optimal point configurations}
\begin{document}
\begin{abstract}
A real symmetric matrix is called $d$-locally positive semidefinite if all of its $d \times d$ principal submatrices are positive semidefinite. We investigate the spectral geometry of $d$-locally positive semidefinite matrices. The set of vectors of eigenvalues of $d$-locally positive semidefinite matrices of size $n \times n$ is fully understood and known to be convex when $d \in \{1,n-1,n\}$ \cite{blekherman2022hyperbolic}. In the smallest remaining case $n=4, d=2$, non-convexity of the set of vectors of eigenvalues was proved in \cite{kozhasov2023eigenvalues}, but even in this case the full description was unknown.
We provide a basic semialgebraic description of the set of vectors of eigenvalues for $n=4, d=2$ by establishing a Fischer-type inequality for $2$-locally positive semidefinite matrices of size $4 \times 4$. We prove non-convexity for all $n \geq 4$ and $2 \leq d \leq n - 2$. 
For $d=2$ we prove a natural linear inequality on a linear slice of eigenvalues of locally positive semidefinite matrices which is tight at exactly two points. This gives a stronger proof of non-convexity in this case.
For $d \in \{2, n - 2\}$, analyzing this slice involves solving certain non-smooth and non-convex min-max point configuration problems in the complex plane, which could be of independent interest.
Similar problems were considered in \cite{nesterenko2024submatrices,sengupta2026submatrices} in the context of matrix decomposition and approximation.
\end{abstract}
\maketitle

\section{Introduction}
Locally positive semidefinite (psd) matrices were introduced in \cite{blekherman2022sparse,blekherman2022hyperbolic} motivated by applications in optimization. A symmetric matrix is $d$-locally psd if all of its $d\times d$ principal submatrices are psd. We write the set of such matrices as $\S^{n, d}$. Its dual cone is the cone of matrices of factor width $d$, which are also important in optimization \cite{boman2005factor}. Locally psd matrices also see applications in quantum coherence \cite{johnston2022absolutely}. The investigation of the geometry of eigenvalues of locally psd matrices was started in \cite{blekherman2022hyperbolic}; it is quite rich and interesting in its own right, although it is difficult to understand in general.

It is well-known that a symmetric matrix is positive semidefinite if and only if its eigenvalues are nonnegative. However, understanding the possible eigenvalues of locally psd matrices turns out to be quite intricate and has connections to spectral graph theory, frame theory\footnote{In forthcoming work \cite{acevedo2026}, we make these connections explicit. We establish that the geometry of 2-locally psd matrices yields connections to spectral graph theory, and the containment of certain boundary points of $H(e_2)$ in $\lambda(\mathcal{S}^{n,2})$ completely encodes the existence of real equiangular tight frames.}, signal processing, and compressed sensing.  

We study the set $\lambda(\mathcal{S}^{n,d})$ of all (unordered) $n$-tuples of eigenvalues of matrices in $\mathcal{S}^{n,d}$. Geometry of $\lambda(\S^{n,d})$ was first considered in \cite{blekherman2022hyperbolic}, where the authors introduced hyperbolic relaxations of $\S^{n,d}$. This led to a convex relaxation of the set $\lambda(\S^{n,d})$ by the hyperbolicity cone $H(e_d)$ of the $d$-th elementary symmetric polynomial in $n$ variables, denoted by $e_d$. 
It is known that $\lambda(\mathcal{S}^{n,n}) = \mathbb{R}^n_{\geq 0}$ and that the hyperbolic relaxation is exact, i.e.,
$\lambda(\mathcal{S}^{n,d}) = H(e_d)$, if and only if $d \in \{1, n - 1, n\}$. Therefore, $\lambda(\mathcal{S}^{n,d})$  is convex in these extreme cases. For all the intermediate values of $d$, the strict inclusion $\lambda(\S^{n,d}) \subsetneq H(e_d)$ holds.

Geometry of these intermediate sets appears to be quite complex.
Using symbolic computation, Kozhasov \cite{kozhasov2023eigenvalues} showed that $\lambda(\mathcal{S}^{4,2})$ is not convex. Specifically, his proof demonstrated that while the endpoints of the line segment connecting $(1,1,0,-\frac{1}{2})$ and $(1,1,-\frac{1}{2},0)$ are contained in $\lambda(\mathcal{S}^{4,2})$, their midpoint is not. However, the complexity of these exact computations restricted this approach to $n = 4$. Moreover, the overall geometry of $\lambda(\S^{4,2})$  was left largely unresolved. 

\subsection*{Main results}

In  Section \ref{sec:nonconv}, we show that the set $\lambda(\S^{n,d})$ is not convex for all $n$ and $2\leq d\leq n-2$. This is done by showing that an explicit convex combination of points in $\lambda(\S^{n,d})$ does not belong to $\lambda(\S^{n,d})$. Our target matrix only has two distinct eigenvalues, and this allows us to reformulate the problem of it belonging to $\lambda(\S^{n,d})$ as a min-max problem on $n \times n$ projection matrices. This reformulation allows us to show that the explicit point does not lie in $\lambda(\S^{n,d})$. We continue by further analyzing the beautiful geometry of $\lambda(\S^{n,d})$.

In Section \ref{sec:S4,2}, we fully characterize the set $\lambda(\S^{4,2})$.
We provide an explicit semialgebraic description of this set in Theorem \ref{theorem: l(S42)}. By analyzing eigenvalues of the form $(1,1,s,t)$, we clearly show how and when convexity breaks, fully generalizing the results in \cite{kozhasov2023eigenvalues} (see Corollary \ref{cor: L42}). The result is derived from a Fischer-type inequality for matrices in $\S^{4,2}$, which could be of independent interest (Proposition \ref{proposition: 2-local Fischer}). We also conclude that the ordered spectrum is not convex for $(n,d)=(4,2)$ (Remark \ref{rem:ordered spectrum non-convex}). Convexity of the ordered spectrum had been raised in \cite[Section~2]{kozhasov2023eigenvalues} and was left open in all cases $d \notin \{1, n - 1, n\}$.

For the case $d=2$, we prove in Section \ref{sec:hypineq} a natural linear inequality for a linear slice of $\lambda(\S^{n,d})$, using ideas similar to Section \ref{sec:nonconv}. This inequality is tight at only two distinct points on the boundary of $\lambda(\S^{n,2})$. This gives us a different, and stronger, proof of non-convexity for this case. We focus our investigation on two-dimensional strata of the symmetric spectra $\lambda(\S^{n,d})$ of the form $(1,\ldots,1,s,t)$. Even in these restricted cases, we encounter interesting and complicated semialgebraic sets. A parallel result for $d=n-2$ was also obtained, but due to the length of the proof it is deferred to the fourth author's PhD thesis \cite{lee2026}. 

We show that identifying some extremal points of $\lambda(\S^{n,d})$, as formulated in Problem \ref{problem:min}, can be reframed and solved as a min-max point configuration problem in the complex plane or Euclidean plane (see Proposition \ref{prop: equivalences} and Theorem \ref{theorem: optimal config d=2} for a solution for $d=2$ and Corollary \ref{cor: optimality of regular n-gon} for a solution for $d=n-2$). 
This geometric reduction closely parallels a structurally related extremal matrix problem initially formulated by Goreinov, Tyrtyshnikov and
Zamarashkin \cite{goreinov1997theory}, according to Nesterenko \cite{nesterenko2024submatrices}. While our work focuses on minimizing the maximum eigenvalue among all $d \times d$ principal submatrices of an $n \times n$ matrix, their problem concerns maximizing the minimal singular value among all $d \times d$ submatrices of an $n \times d$ matrix. 
For $d=2$ this problem was recently solved \cite{sengupta2026submatrices}. 

\subsection{Notation and Terminology}

Let $\Sym_n$ denote the vector space of $n\times n$ real symmetric matrices. For $X\in \Sym_n$, we write $X\succeq 0$ if $X$ is \emph{positive semidefinite} (psd), and we denote the convex cone of psd matrices by $\S^n_+$. We sometimes write $|M|$ for the determinant of a square matrix $M$, and we write $A \sim B$ if two matrices $A,B$ are similar. Also, if $M$ is a symmetric matrix and $D$ is an invertible diagonal matrix, then we will call the matrix $DMD$ a \emph{diagonal scaling} of $M$.

For each integer $d \in \{1, \ldots, n\}$, we say a matrix $M\in\Sym_n$ is $d$-\emph{locally positive semidefinite} if every $d\times d$ principal submatrix of $M$ is psd. The set of all $d$-locally psd matrices is denoted by $\S^{n,d}$. We observe that a symmetric matrix is $d$-locally psd if and only if all principal minors of size at most $d$ are nonnegative.

If a symmetric matrix $M$ has eigenvalues (spectrum) $\{\lambda_1,\ldots,\lambda_n\}$, we refer to the $n$-tuple $\lambda = (\lambda_1,\ldots,\lambda_n)$, and any of its permutations, as a \emph{symmetric spectrum} of $M$.
Our main objects of study are the sets of symmetric spectra of $d$-locally psd matrices
\[\lambda(\S^{n,d}):=\{(\lambda_1,\dots,\lambda_n)\in\RR^n\, \suchthat \,\exists X\in\S^{n,d}:X\sim\diag(\lambda_1,\dots,\lambda_n)\},\]
and their two-dimensional slices
\[
\mathcal{L}_{n, d}:= \left\{(s, t)\,\suchthat\,(1, \ldots, 1, s, t) \in \lambda(\S^{n, d})\right\}.
\]

Let $e_d^n$ denote the $d$-th elementary symmetric polynomial in $n$ variables, given by
\[ e_d^n(x_1,\ldots,x_n) = \sum_{1 \leq i_1 < \ldots < i_d \leq n} x_{i_1}\cdots x_{i_d} .
\]
Most of the time we will suppress the number of variables $n$ and simply write $e_d$, since the number of variables is usually clear from context.
The \emph{hyperbolicity cone} of $e_d$ with respect to the all-ones vector $\mathbf{1} \in \RR^n$ is
\[
H(e_d)=\{x\in\RR^n\,\suchthat\, e_1(x)\ge0,\dots,e_d(x)\ge0\}.
\]

For indexing, let $[n]=\{1,\ldots,n\}$ and let $\binom{[n]}{d}$ denote the set of all $d$-element subsets of $[n]$. For a matrix $X \in \RR^{m \times n}$ and a subset $J \in \binom{[n]}{d}$, we write $X_J$ for the $m \times d$ submatrix of $X$ consisting of the columns indexed by $J$. For a square matrix $A \in \RR^{n \times n}$ we write $A_J$ for the $d \times d$ principal submatrix of $A$ whose columns and rows are indexed by $J$, e.g., $A_{\{1,2\}}$ is the $2\times 2$ submatrix of $A$ defined by taking its first two rows and columns. We will also write $J^c = [n] \setminus J$.

For a square matrix $X$, we write $\rho(X)$ for its spectral radius (the maximum absolute value of its eigenvalues), and $\|X\|_F = \sqrt{\tr (XX^\top)}$ for its Frobenius norm. 
For $X \in \Sym_n$ we denote the minimal eigenvalue of $X$ by $\lambda_{\min}(X)$ and the maximal eigenvalue of $X$ by $\lambda_{\max}(X)$.

Throughout the paper, indices will be labeled cyclically in the context of summations, e.g., we will assume $x_{n + 1}=x_1$ when writing $\sum_{k=1}^n h(x_k, x_{k + 1})$.

\subsection{Preliminaries}

Note that each $\S^{n,d}$ is a convex cone. This implies that the set $\lambda(\S^{n,d})$ is star-shaped with respect to $\RR_{\geq 0}^n$ (see \cite[Proposition~1.4]{kozhasov2023eigenvalues}). 
Moreover, if $M \in \S^{n,d}$ and $D \in \Sym_n$ is an invertible diagonal matrix, then the diagonal scaling $DMD$ is also $d$-locally psd. 

As observed in \cite{blekherman2022hyperbolic}, the matrix 
\begin{align}\label{eq: G(n,d)}
G(n, d):= \frac{d}{d-1}I - \frac{1}{d-1}\mathbf{1}\mathbf{1}^\top
\end{align}
lies in $\S^{n, d}$ but is not psd for $2\leq d < n$. In particular, all diagonal scalings $DG(n,d)D$ are $d$-locally psd. 

For an $n\times n$ matrix $A$ with eigenvalues $\lambda_1,\ldots,\lambda_n$, the sum of its $d \times d$ principal minors is equal to the $d$-th elementary symmetric polynomial in the eigenvalues of $A$:
\begin{align}\label{eq: det(A) in eigenvalues}
\sum_{J \in \binom{[n]}{d}}\det(A_J) = e_d(\lambda_1,\ldots,\lambda_n).    
\end{align}

Since principal minors of size at most $d$ are nonnegative for all $d$-locally positive semidefinite matrices we have that $\lambda(\S^{n,d})\subseteq H(e_d)$. The sets $H(e_d)$ are hyperbolicity cones and are thus convex. Furthermore, the boundary of $H(e_d)$ is precisely its intersection with the hypersurface defined by $e_d=0$. In \cite{MR4681290} the authors considered generalizations of these hyperbolic relaxations to more general sums of linear polynomial in principal minors.

Since $\S^{n,1}\supsetneq\S^{n,2}\supsetneq\dots\supsetneq\S^{n,n}$, the sets of symmetric spectra naturally satisfy the inclusions 
\[\lambda(\S^{n,1})\supseteq \lambda(\S^{n,2})\supseteq \dots\supseteq \lambda(\S^{n,n}).\]
Moreover, each inclusion is strict \cite{blekherman2022hyperbolic}.
If $\lambda \in \lambda(\S^{n, d})$, then $\lambda$ has at most $n-d$ negative entries. This follows from Descartes' rule of signs; since $\lambda \in H(e_d)$, we have $e_1(\lambda),\dots ,e_d(\lambda)\ge0$, so the polynomial \[(x+\lambda_1)\cdots(x+\lambda_n)=x^n+e_1(\lambda)x^{n-1}+\dots+e_d(\lambda)x^{n-d}+\dots+e_{n-1}(\lambda)x+e_n(\lambda)\] has at most $n-d$ sign changes.

\section{Basic semialgebraic description of \texorpdfstring{$\lambda(\mathcal{S}^{4,2})$}{lambda(S4,2)} } \label{section: 2-locally psd} \label{sec:S4,2}

In this section, we give basic semialgebraic descriptions of $\lambda(\S^{4, 2})$ and its slice $\mathcal{L}_{4, 2}$. The descriptions are given in Theorem \ref{theorem: l(S42)} and Corollary \ref{cor: L42} respectively. The key insight is the following Fischer-type inequality for matrices in $\S^{4,2}$, which may be interesting in itself.

\begin{proposition}\label{proposition: 2-local Fischer}
    Let $A=(a_{ij})_{1 \leq i,j \leq 4} \in \S^{4, 2}$. We have
    \begin{align*}
       |A_{\{1, 2\}}| |A_{\{3, 4\}}|+ |A_{\{1, 3\}}| |A_{\{2, 4\}}|+ |A_{\{1, 4\}}||A_{\{2, 3\}}| \ge |A|~.
    \end{align*}
    For nonsingular $A$, equality is achieved if and only if all $2\times2$ principal minors in a $4$-cycle vanish and the product of the entries in the same $4$-cycle is negative, e.g. $|A_{\{1,2\}}|=|A_{\{2, 3\}}|=|A_{\{3, 4\}}|=|A_{\{1,4\}}|=0$ and $a_{12}a_{23}a_{34}a_{14}<0$.
\end{proposition}
\begin{proof} 
    If $A$ has a zero diagonal entry, say $a_{11}=0$, then since $A$ is $2$-locally psd, every entry in the first row and column must be zero. So, $|A|=0$ and $|A_{\{1,2\}}|=|A_{\{1,3\}}|=|A_{\{1,4\}}|=0$, and the inequality holds. Assume now that all diagonal entries of $A$ are positive. By a diagonal scaling we may assume without loss of generality that 
    \[
    A=\begin{bmatrix}
        1 & a_{12} & a_{13} & a_{14}\\
        a_{12} & 1 & a_{23} & a_{24}\\
        a_{13} & a_{23} & 1 & a_{34}\\
        a_{14} & a_{24} & a_{34} & 1
    \end{bmatrix}.
    \]
    Since $A$ is $2$-locally psd, we have $a_{ij}\in[-1,1]$.
    
    Expanding $\Delta:=\frac12(\text{LHS}-\text{RHS})$ of the inequality to be proven, we obtain
    \[
    1-a_{12}a_{13}a_{23}-a_{12}a_{14}a_{24}-a_{13}a_{14}a_{34}-a_{23}a_{24}a_{34}+a_{13}a_{14}a_{23}a_{24}+a_{12}a_{14}a_{23}a_{34}+a_{12}a_{13}a_{24}a_{34},
    \]
    which is multiaffine, so it achieves its minimum at a vertex of the hypercube $[-1,1]^6$. Setting $p=a_{12}a_{13}a_{23}$, $q=a_{12}a_{14}a_{24}$, $r=a_{13}a_{14}a_{34}$, modulo $a_{ij}\in\{-1,1\}$, the above expression equals \[
    1-p-q-r-pqr+pq+pr+qr=(1-p)(1-q)(1-r)
    \] which is nonnegative since $p,q,r\in\{-1,1\}$.

    Since $\Delta$ is multiaffine and nonnegative on the hypercube $H:=[-1,1]^6$, its vanishing set in $H$ must consist precisely of the union of the faces of $H$ on which $\Delta$ vanishes. Analyzing the vertex condition $(1-p)(1-q)(1-r)=0$, we conclude that the vanishing set of $\Delta$ on $H$ is the union of twenty four $2$-dimensional faces and sixteen $3$-dimensional faces of $H$.

    The $24$ $2$-dimensional faces correspond to the $8$ possible assignments of $\pm1$ to each edge of the three $4$-cycles of $K_4$, such that the product of the edges is $-1$. For example, if $a_{12}=a_{23}=a_{34}=1$ and $a_{14}=-1$, then $\Delta=0$ for any $a_{13}$ and $a_{24}$.

    The $16$ $3$-dimensional faces correspond to the $4$ possible assignments of $\pm1$ to each of the four $3$-cycles of $K_4$, such that the product of the edges is $1$. For example, if $a_{12}=a_{13}=a_{23}=1$, then $\Delta=0$ for any $a_{14},a_{24},a_{34}$. Observe that in each of these faces $|A|=0$.

    So, if equality is achieved for a nonsingular $A\in\mathcal{S}^{4,2}$, then $A$ can be diagonally scaled, up to a permutation of rows and columns, to a matrix $M$ with $m_{12}=m_{23}=m_{34}=1$ and $m_{14}=-1$, i.e., $|A_{\{1,2\}}|=|A_{\{2,3\}}|=|A_{\{3,4\}}|=|A_{\{1,4\}}|=0$ and $a_{12}a_{23}a_{34}a_{14}<0$.
\end{proof}

Using Proposition \ref{proposition: 2-local Fischer}, we establish an inequality for matrices in $\S^{4,2}$ that compares two coefficients of the characteristic polynomial (compare with \eqref{eq: det(A) in eigenvalues}). 

\begin{proposition}\label{prop: e_2^2>=4det}
Let $A \in \S^{4, 2}$. We have
\[
e_2(A)^2\ge 4\det(A),
\]
where $e_2(A) = \sum_{1 \leq i < j \leq 4} |A_{\{i,j\}}|$.
\end{proposition}
\begin{proof} If $A\in\S^{4,2}$, then $|A_{\{i,j\}}|\ge0$ for $1\le i<j\le4$, hence
\begin{align*}
e_2(A)^2&=\left(|A_{\{1,2\}}|+|A_{\{3,4\}}|+|A_{\{1,3\}}|+|A_{\{2,4\}}|+|A_{\{1,4\}}|+|A_{\{2,3\}}|\right)^2\\
 &\ge \left(|A_{\{1,2\}}|+|A_{\{3,4\}}|\right)^2+\left(|A_{\{1,3\}}|+|A_{\{2,4\}}|\right)^2+\left(|A_{\{1,4\}}|+|A_{\{2,3\}}|\right)^2\\
 &\ge 4|A_{\{1,2\}}||A_{\{3,4\}}|+4|A_{\{1,3\}}||A_{\{2,4\}}|+4|A_{\{1,4\}}||A_{\{2,3\}}|.
\end{align*}
The first inequality follows since each $2\times 2$ principal minor is nonnegative, and the second inequality follows from the AM-GM inequality. The conclusion follows from Proposition \ref{proposition: 2-local Fischer}.
\end{proof}

\begin{remark}\label{rem: 4-cycle}
    Let 
    \[
    u=|A_{\{1,2\}}|+|A_{\{3,4\}}|, \quad v=|A_{\{1,3\}}|+|A_{\{2,4\}}|, \quad w=|A_{\{1,4\}}|+|A_{\{2,3\}}|.
    \]For equality to hold in Proposition \ref{prop: e_2^2>=4det} we need 
    \begin{align*}
    (u+v+w)^2=u^2+v^2+w^2 \quad\Leftrightarrow \quad uv+vw+wu=0,
    \end{align*}
     and this is equivalent to at least two of $u,v,w$ vanishing (since $u,v,w$ are all nonnegative). Say $u=w=0$; this is equivalent to $|A_{\{1,2\}}|=|A_{\{2,3\}}|=|A_{\{3,4\}}|=|A_{\{1,4\}}|=0$. Also observe that, since $(|A_{\{1,3\}}|+|A_{\{2,4\}}|)^2=4|A_{\{1,3\}}||A_{\{2,4\}}|$, we must have $|A_{\{1,3\}}|=|A_{\{2,4\}}|$.
\end{remark}

We use Remark \ref{rem: 4-cycle} to determine the boundary of $\lambda(\mathcal{S}^{4,2})$. The matrices
\begin{align}\label{eq: A(x)}
A(x)=\begin{bmatrix}
    1 & 1 & x & -1\\
    1 & 1 & 1 & x\\
    x & 1 & 1 & 1\\
    -1 & x & 1 & 1
    \end{bmatrix}
\end{align}
with $x\in[-1,1]$ are $2$-locally psd and satisfy equality in Proposition \ref{prop: e_2^2>=4det}. Moreover, a diagonal scaling of $A(x)$ produces a matrix in $\S^{4, 2}$ whose symmetric spectrum $\lambda$ satisfies $e_2(\lambda)^2 = 4e_4(\lambda)$. In the proof of Theorem \ref{theorem: l(S42)}, we will show that in fact every point in $H(e_2)$ on this hypersurface with two negative coordinates is a symmetric spectrum of a matrix that is a diagonal scaling of $A(x)$ for some $x\in[-1,1]$. We will need the following two auxiliary results.

\begin{lemma}\label{lemma:positive dominates negative}
If $\lambda = (a, b, s, t)$ satisfies $a, b \geq 0$, $s, t < 0$, $e_1(\lambda) \geq 0$, and $e_2(\lambda) \geq 0$, then $a, b > 0$, $a + s > 0$, and $b + t > 0$.
\end{lemma}

\begin{proof}
    We will show that the sum and product of $a + s$ and $b + t$ are each nonnegative. Indeed,
    \[
    (a + s) + (b + t) = e_1(\lambda)\geq 0,
    \]
    and since $as + bt \leq 0$, we have
    \[
    (a + s)(b + t) \geq (a + s)(b + t) + as + bt = e_2(\lambda) \geq 0,
    \]
    as claimed. It follows that $a + s$ and $b + t$ are both nonnegative, and thus $a \geq -s > 0$ and $b \geq -t > 0$. This also means $as + bt < 0$, so the product inequality is actually strict, and hence both $a + s$ and $b + t$ are positive.
\end{proof}

\begin{lemma}\label{lemma: e2-e4 boundary}
    Assume $\lambda = (a, b, s, t)$ satisfies $a, b > 0$, $s, t < 0$, $e_1(\lambda) \geq 0$, $e_2(\lambda) \geq 0$, and $e_2(\lambda)^2 \geq 4e_4(\lambda)$. For all $\tau \geq 0$, define
    \[
    a_\tau = a - \tau, \; \; b_\tau = b - \tau, \; \; s_\tau = s - \tau, \; \; t_\tau = t - \tau, \; \; \lambda_\tau = (a_\tau, b_\tau, s_\tau, t_\tau).
    \]
    Then there is $0 \leq \tau \leq \frac12\min\{a + s, b + t\}$ such that $a_\tau, b_\tau > 0$, $s_\tau, t_\tau < 0$, $e_1(\lambda_\tau) \geq 0$, $e_2(\lambda_\tau) \geq 0$, and $e_2(\lambda_\tau)^2 = 4e_4(\lambda_\tau)$.
\end{lemma}

\begin{proof}
    For $\tau \geq 0$, define 
    \[
    f(\tau) = e_2(\lambda_\tau) - 2\sqrt{e_4(\lambda_\tau)}.
    \]
    We have \(f(0) \geq 0\) by assumption. Also, let
    \[
    \tau_* =  \frac12\min\{a + s, b + t\}.
    \]
    By Lemma \ref{lemma:positive dominates negative}, this number is positive. Moreover, $a_{\tau_*}$ and $b_{\tau_*}$ are both positive, and $s_{\tau_*}$ and $t_{\tau_*}$ are both negative. Finally, one of 
    $a_{\tau_*} + s_{\tau_*}$ and $b_{\tau_*} + t_{\tau_*}$
    is zero, and the other is nonnegative. Thus, by using the identity
    \begin{align}
    f(\tau) &= (a_\tau + s_\tau)(b_\tau + t_\tau) + a_\tau s_\tau + b_\tau t_\tau - 2\sqrt{a_\tau b_\tau s_\tau t_\tau} \nonumber \\
    &= (a_\tau + s_\tau)(b_\tau + t_\tau) - (\sqrt{-a_\tau s_\tau} + \sqrt{-b_\tau t_\tau})^2, \label{eq: e2e4}
    \end{align}
    we see that $f(\tau_*) \leq 0$. Since $f$ is continuous on $[0, \tau_*]$, the Intermediate Value Theorem guarantees the existence of $\tau \in [0, \tau_*]$ such that $f(\tau) = 0$. For this $\tau$, we also have $a_\tau, b_\tau > 0$, $s_\tau, t_\tau < 0$, and
    \[e_1(\lambda_\tau) = (a_\tau + s_\tau) + (b_\tau + t_\tau) \geq 0, \quad e_2(\lambda_\tau) = 2\sqrt{e_4(\lambda_\tau)} \geq 0.\]
\end{proof}

In the following Theorem and Corollary we give semialgebraic descriptions of $\lambda(\S^{4,2})$ and its slice with the affine subspace $\{(1,1,s,t) \, \suchthat \, s,t \in \RR\}$ (see also Figures \ref{fig:spectra} and \ref{fig:boundary42}).

\begin{theorem}\label{theorem: l(S42)}
    $\lambda(\mathcal{S}^{4,2})$ is the basic semialgebraic set given by the $\lambda\in\RR^4$ satisfying the inequalities 
    \begin{align*}
        e_1(\lambda)\ge0,\quad
        e_2(\lambda)\ge0,\quad
        e_2(\lambda)^2\ge4e_4(\lambda).
    \end{align*}
\end{theorem}
\begin{proof}
    Assume that $\lambda \in \lambda(\S^{4, 2})$. Since $\lambda(\S^{4, 2}) \subseteq H(e_2)$, the inequalities $e_1(\lambda)\ge0$ and $e_2(\lambda)\ge0$ are satisfied. If $A \in \S^{4, 2}$ has symmetric spectrum $\lambda$, then by Proposition \ref{prop: e_2^2>=4det}, $e_2(\lambda)^2\ge4e_4(\lambda)$ is also satisfied, since $e_2(\lambda)=e_2(A)$ and $e_4(\lambda)=\det(A)$ by \eqref{eq: det(A) in eigenvalues}. 
    
    We now show that the inequalities are sufficient. Assume that $\lambda$ satisfies the given inequalities. If $\lambda$ has at most one negative coordinate, then by \cite[Corollary~1.5]{kozhasov2023eigenvalues} we know that $\lambda \in \lambda(\S^{4, 2})$. Since three negative coordinates are not possible, we may assume $\lambda$ has exactly two negative coordinates.
    
    Put $\lambda = (a, b, s, t)$, where $a, b \geq 0$ and $s, t < 0$. By Lemma \ref{lemma:positive dominates negative}, we obtain that $a, b > 0$, $a + s > 0$, and  $b + t > 0$. By Lemma \ref{lemma: e2-e4 boundary}, we may push $\lambda$ in the direction of $(-1, -1, -1, -1)$ until the vector lies on the hypersurface $e_2(\lambda)^2 = 4e_4(\lambda)$ while preserving the signs of the coordinates. So by the star-shapedness of $\lambda(\S^{4, 2})$ with respect to the nonnegative orthant, it suffices to prove inclusion of the hypersurface $e_2(\lambda)^2 = 4e_4(\lambda)$ intersected with the set of points with sign pattern $(+,+,-,-)$ in $\lambda(\S^{4, 2})$.
    
    Indeed, assume $\lambda = (a, b, s, t)$ satisfies $a,  b > 0$, $s, t < 0$, $e_1(\lambda) \geq 0$, $e_2(\lambda) \geq 0$, and $e_2(\lambda)^2 = 4e_4(\lambda)$. Again, we obtain that $a + s > 0$ and $b + t > 0$, and by \eqref{eq: e2e4} we have $
    (a + s)(b + t) = (\sqrt{-as} + \sqrt{-bt})^2$.
    Now let
    \[
    \alpha = \sqrt{\frac{a + s}2}, \quad \beta  = \sqrt{\frac{b + t}2}, \quad x = \frac{\sqrt{-bt} - \sqrt{-as}}{\sqrt{-as} + \sqrt{-bt}},
    \]
    where we have that $x \in (-1, 1)$. Also let
    \[
    D = \diag(\alpha, \beta, \beta, \alpha), \quad M = DA(x)D.
    \]
    We see that $M$ is 2-locally psd. We claim that $M$ has the desired spectrum $\lambda$.

    We observe that the orthogonal subspaces
    \[
    W_+ := \{(u, v, v, u) : u, v \in \RR\}, \quad W_- := \{(u, v, -v, -u) : u, v \in \RR\}
    \]
    are invariant under both $D$ and $A(x)$, and hence also under $M$. On $W_+$, the matrix representation of $M$ over the ordered basis $\{(1, 0, 0, 1), (0, 1, 1, 0)\}$ can be computed as
    \[
    \begin{bmatrix}
        0 & (1 + x)\alpha\beta \\
        (1 + x)\alpha\beta & 2\beta^2
    \end{bmatrix},
    \]
    which has characteristic polynomial
    \begin{align*}
    &\lambda^2 - 2\beta^2\lambda - (1 + x)^2\alpha^2\beta^2 \\
    &= \lambda^2 - (b + t)\lambda - \frac{-4bt}{(\sqrt{-as} + \sqrt{-bt})^2} \cdot \frac{a + s}2 \cdot \frac{b + t}2 \\
    &= \lambda^2 - (b + t)\lambda + bt = (\lambda - b)(\lambda - t).
    \end{align*}
    Likewise, the matrix of $M$ on $W_-$ over the ordered basis $\{(1, 0, 0, -1), (0, 1, -1, 0)\}$ is
    \[
    \begin{bmatrix}
        2\alpha^2 & (1 - x)\alpha\beta \\
        (1 - x)\alpha\beta & 0
    \end{bmatrix},
    \]
    and a similar computation yields the characteristic polynomial $(\lambda - a)(\lambda - s)$.
    Thus, we recover the eigenvalues of $M$ as $a, b, s, t$.
\end{proof}

\begin{figure}[htbp]
  \centering
  \begin{subfigure}[t]{0.31\textwidth}
    \centering
    \includegraphics[height=30mm]{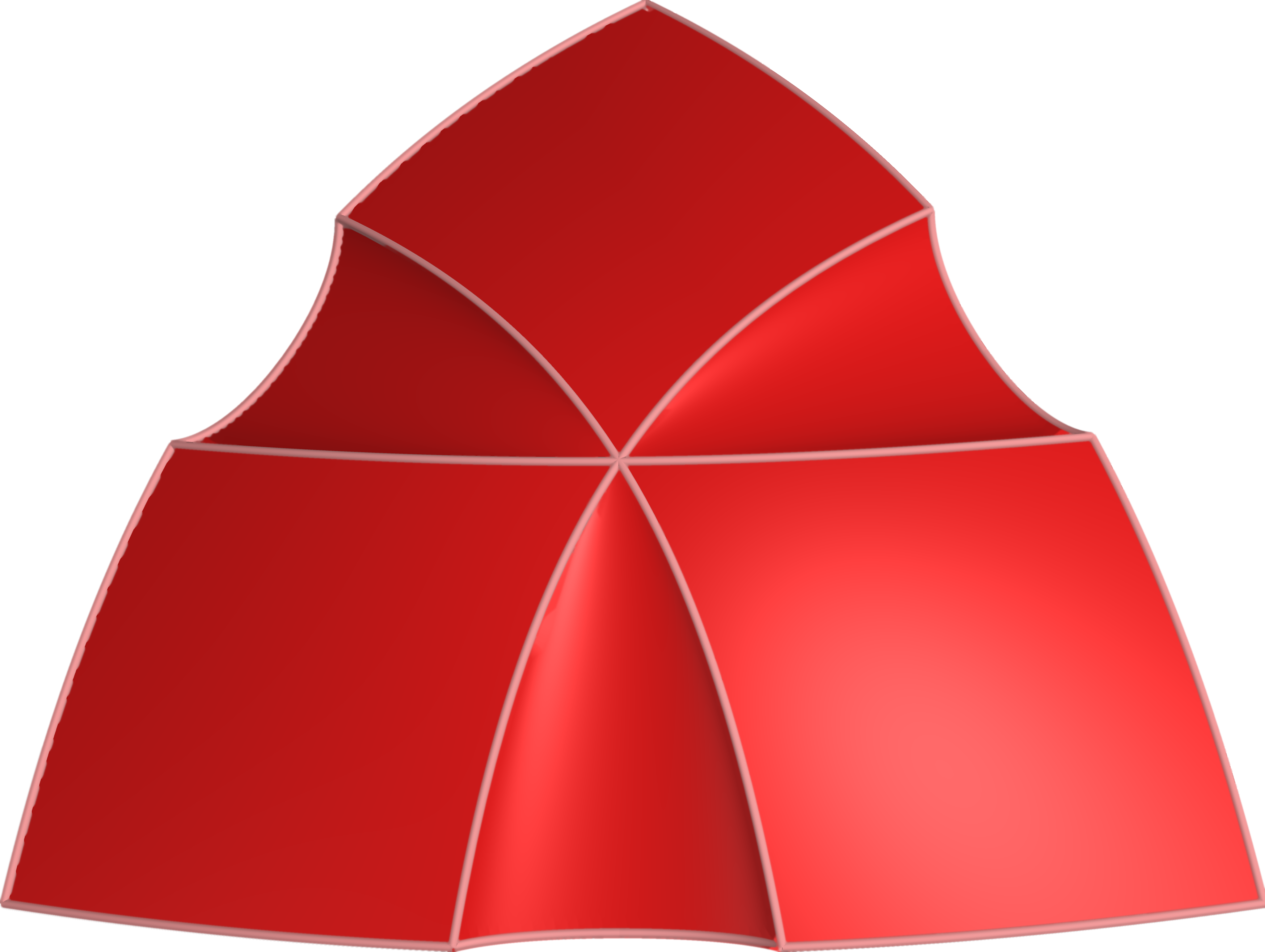}
    \caption{Projection of $\lambda(\mathcal{S}^{4,2}) \cap \{\lambda \mid
      1 = \lambda_1 \ge \lambda_2 \ge \lambda_3 \ge \lambda_4\}$
      onto the last three eigenvalues.}
    \label{fig:spec-A}
  \end{subfigure}
  \hfill
  \begin{subfigure}[t]{0.31\textwidth}
    \centering
    \includegraphics[height=30mm]{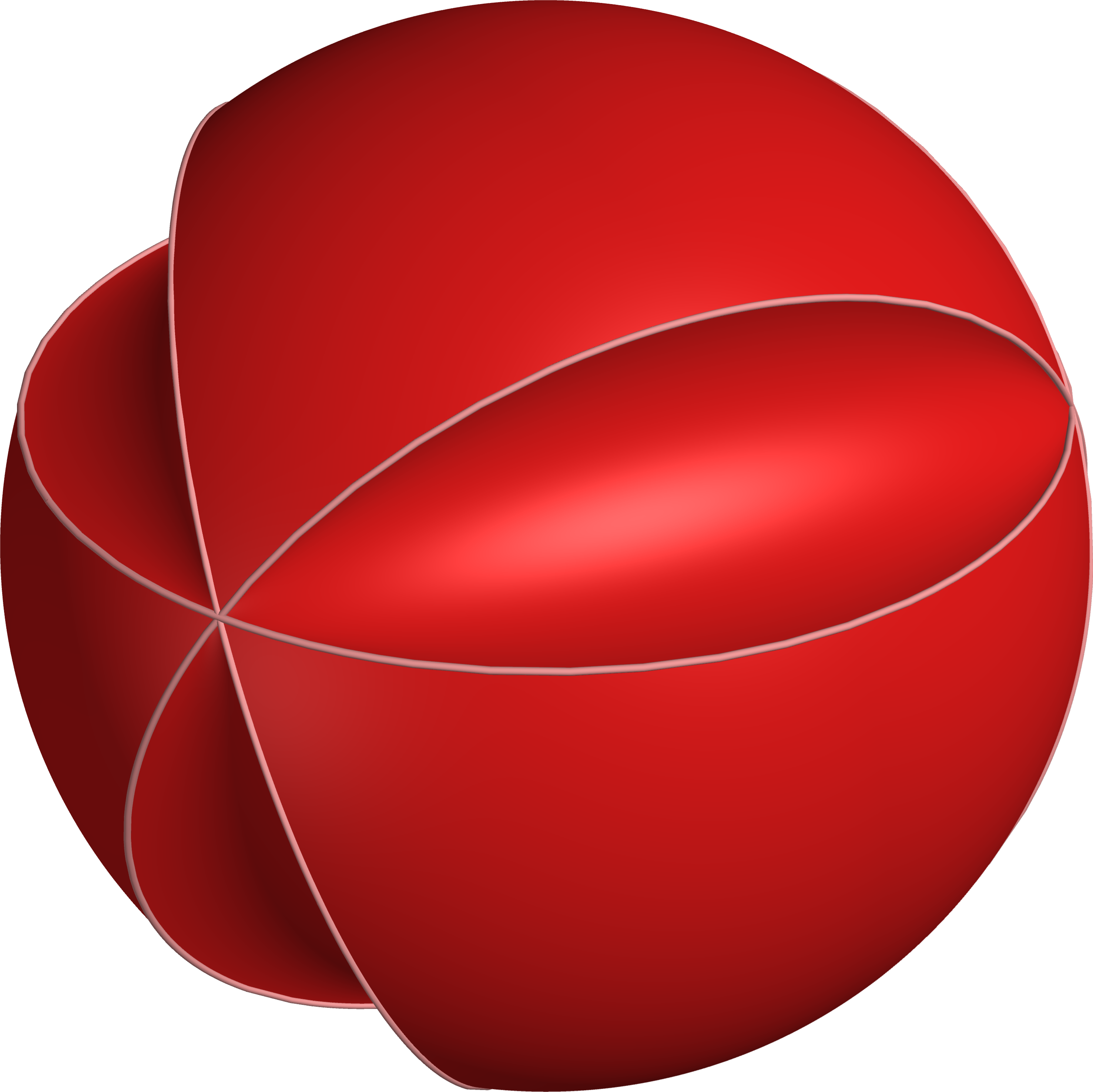}
    \caption{$\{\lambda \in \lambda(\mathcal{S}^{4,2}) \mid
      \sum_{i=1}^4 \lambda_i = 1\}$, drawn in the
      hyperplane $\sum_i \lambda_i = 1$.}
    \label{fig:spec-B}
  \end{subfigure}
  \hfill
  \begin{subfigure}[t]{0.31\textwidth}
    \centering
    \includegraphics[height=30mm]{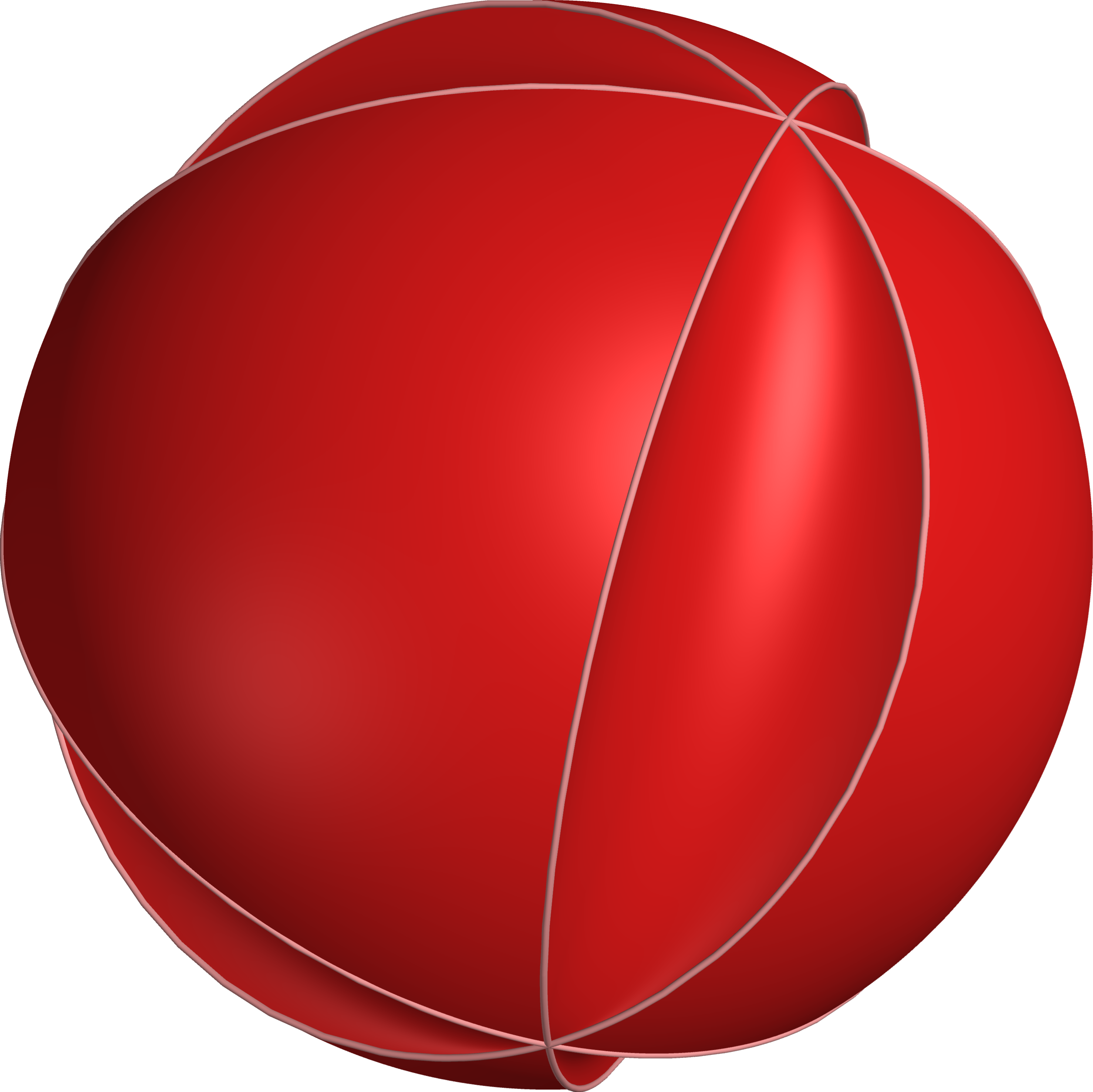}
    \caption{The section in (\subref{fig:spec-B}) seen from a
      different viewpoint.}
    \label{fig:spec-C}
  \end{subfigure}
  \caption{Sections of $\lambda(\mathcal{S}^{4,2})$.}
  \label{fig:spectra}
\end{figure}
\begin{figure}[ht!]
\centering
\resizebox{4cm}{!}{
\includegraphics[width=0.2\textwidth]{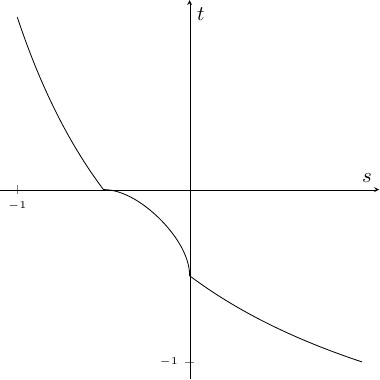}
}
\caption{Boundary of $\mathcal{L}_{4,2}$.}
  \label{fig:boundary42}
\end{figure}

The convex parts of the boundary in Figure \ref{fig:spec-B} and \ref{fig:spec-C} correspond to the boundary of the unit ball. The four singular points are vertices of a regular tetrahedron inscribed in the ball. The curves on the sphere joining the singular points come from intersecting the supporting hyperplanes of the regular tetrahedron with the sphere. An interactive orthogonal projection of this slice onto a coordinate hyperplane can be found in \url{https://jgacevedoh.github.io/animations/eigen2lpsd4.html}.

The following description of $\mathcal{L}_{4,2}$ follows directly from Theorem \ref{theorem: l(S42)} by evaluating at $(1,1,s,t)$.

\begin{corollary}\label{cor: L42}
    $\mathcal{L}_{4,2}$ is the basic semialgebraic set given by the inequalities
    \begin{align*}
        2 + s+t\ge 0,\quad
        1+2s+2t+st\ge0,\quad
        (1+2s+2t+st)^2\ge4st.
    \end{align*}
\end{corollary}
\begin{corollary}\label{cor:ordered spectrum}
The sets $\mathcal{L}_{4, 2}$ and $\lambda(\S^{4,2})$ are not convex.
\end{corollary}

\begin{remark}\label{rem:ordered spectrum non-convex}
It was asked in \cite[Section~2]{kozhasov2023eigenvalues} whether the \emph{ordered spectrum} $\lambda_{\leq}(\S^{n,d}) := \{ \lambda \in \lambda(\S^{n,d}) \, \suchthat \, \lambda_1 \leq  \dots  \leq\lambda_n\}$ is not convex for some $(n,d)$. Corollary \ref{cor: L42} implies that the ordered spectrum is indeed not convex for $(n, d) = (4, 2)$ as illustrated by Figure \ref{fig:boundary42}.
\end{remark}

\section{Non-convexity of \texorpdfstring{$\lambda(\mathcal{S}^{n,d})$}{lambda(Sn,d)} for \texorpdfstring{$2 \leq d \leq n-2$}{2<=d<=n-2} } \label{sec:nonconv}
In this section we show that the sets $\lambda(\S^{n,d})$ are not convex for $2 \leq d \leq n-2$. We achieve this by proving that the average of the $n-d$ points
\begin{align}\label{eq:n-d points}
    \left(\underbrace{1, \ldots, 1}_{d}, \underbrace{0, \ldots, 0}_{n-d-1}, -\frac1d\right), \left(\underbrace{1, \ldots, 1}_{d}, \underbrace{0, \ldots, 0}_{n-d-2}, -\frac1d, 0\right),  \cdots, \left(\underbrace{1, \ldots, 1}_{d},  -\frac1d,\underbrace{0, \ldots, 0}_{n-d-1}\right)
\end{align}
is not contained in $\lambda(\S^{n,d}).$ Observe that each of the $n-d$ points lies in $\lambda(\S^{n,d})$ since $(1,\dots,1,-\frac{1}{d})$ is the spectrum of the matrix $\frac{d-1}{d} G(d + 1, d)$ which lies in $\S^{d + 1, d}$.
The block diagonal matrix in $\Sym_n$ with blocks $\frac{d-1}{d} G(d + 1, d)$ and $0 \in \Sym_{n-d-1}$ lies in $\S^{n,d}$ and its eigenvalues are the coordinates of the points in \eqref{eq:n-d points}. So the $n-d$ points are indeed in $\lambda(\S^{n,d})$.

For $1 \leq d \leq n - 1$, we will consider the value
\begin{align}\label{eq:alpha nd}
    \alpha_{n,d}:=\max_{\substack{P\in\Sym_n\\ P=P^2\\ \rank P=d}}\min_{S\in\binom{[n]}d}\lambda_{\min}(P_S).
\end{align}
Note that the maximum in \eqref{eq:alpha nd} is attained since the set of rank-$d$ orthogonal projection matrices in $\Sym_n$ is compact and $\lambda_{\min}$ is a continuous function on $\Sym_d$.

The following observation demonstrates the relevance of $\alpha_{n,d}$ for proving non-convexity.

\begin{lemma}\label{lem: alpha and epsilon}
    Let $1 \leq d \leq n-1$ and $\varepsilon > 0$. Then $(\underbrace{1,\dots,1}_d,\underbrace{-\varepsilon,\dots,-\varepsilon}_{n-d}) \in \lambda(\S^{n,d})$ if and only if $\alpha_{n,d} \geq \frac{\varepsilon}{1+\varepsilon}$.
\end{lemma}
\begin{proof}
Let $M \in \Sym_n$ with $M \sim \diag(1,\dots,1,-\varepsilon,\dots,-\varepsilon)$ with $d$-many ones. By the spectral theorem we have $M=P-\varepsilon(I-P)=(1+\varepsilon)P-\varepsilon I$, where $P$ is an orthogonal projection matrix onto the $1$-eigenspace of $M$.   
For $S\in \binom{[n]}{d}$ we have $M_S=(1+\varepsilon)P_S-\varepsilon I_d$ which is psd if and only if $\lambda_{\min}(P_S) \geq \frac{\varepsilon}{1+\varepsilon}$. 
So $(1,\dots,1,-\varepsilon,\dots,-\varepsilon) \in \lambda(\S^{n,d})$ if and only if there exists an orthogonal projection matrix $P$ of rank $d$ such that $\min_{S \in \binom{[n]}{d}} \lambda_{\min}(P_S) \geq \frac{\varepsilon}{1+\varepsilon}$. This proves the claim. 

\end{proof}

For instance, every orthogonal projection matrix of rank $1$ has the form $P=uu^\top$ for a unit vector $u \in \RR^n$. A $1 \times 1$-principal submatrix $P_{\{i\}}$ is the diagonal entry $u_i^2$. Since $\min_{1 \leq i \leq n}u_i^2 \leq \frac{1}{n}$ with equality for $u =(\frac{1}{\sqrt{n}},\dots,\frac{1}{\sqrt{n}})$ we have \[\alpha_{n,1}=\max_{||u||_2=1} \min_{1 \leq i \leq n} u_i^2 = \frac{1}{n}.\]   

We make several preliminary observations. 

\begin{lemma} \label{lemma: alpha n,d}
For all $1 \leq d \leq n - 1$, we have $\alpha_{n, d} = \alpha_{n, n - d}$.
\end{lemma}

\begin{proof}
Let $P$ be a rank-$d$ orthogonal projection in $\RR^n$. Observe that $I-P$ is an orthogonal projection of rank $n-d$, and every orthogonal projection of rank $n-d$ in $\RR^n$ can be written as $I-P'$ for some orthogonal projection $P'$ of rank $d$. Thus, if $Q=I-P$ then it is enough to show that $\lambda_{\min}(P_S)=\lambda_{\min}(Q_T)$ for all $S\in\binom{[n]}d$ and $T=S^c$.
    
Let $P=A^\top A$ with $AA^\top=I_d$, $A\in\RR^{d\times n}$. Observe $A_S A_S^\top+A_TA_T^\top=I_d$ and that $P_S=A_S^\top A_S$ and $A_S A_S^\top=I_d-A_T 
A_T^\top$ share the same spectrum because $A_S$ is a square matrix. 
    
But $A_TA_T^\top$ and $A_T^\top A_T$ share the same nonzero eigenvalues and both matrices are psd. So, the two matrices have the same largest eigenvalue. Since $Q_T=I_{n-d}-A_T^\top A_T$ we obtain
\begin{align*}
&\lambda_{\min}(P_S)=\lambda_{\min}(I_d-A_T A_T^\top)=1-\lambda_{\max}(A_TA_T^\top)\\
&=1-\lambda_{\max}(A_T^\top A_T) = \lambda_{\min}(I_{n-d}-A_T^\top A_T) = \lambda_{\min}(Q_T).
\end{align*}
\end{proof}

\begin{lemma}\label{lemma: min gap}
    Let $n \geq 2$ be an integer and $z_1,\dots,z_n\in\CC$. Then
    \[
    \min_{i<j}|z_i|+|z_j|-|z_i+z_j|\le\frac2n\left(1-\cos\frac{\pi}n\right)\sum_{k=1}^n|z_k|,
    \]
and equality holds if and only if all $z_k$ are zero or $z_1,\dots,z_n$ are the vertices of an origin-centered regular $n$-gon.     
\end{lemma}
Note that up to permutation the vertices of an origin-centered regular $n$-gon can be written as $z_k = r \mathrm{e}^{\mathrm{i}(\theta+\frac{2\pi  k}{n}})$ for $1 \leq k \leq n$ and some $r > 0$ and $\theta \in \RR$.
\begin{proof}
Let $L$ and $R$ denote the left- and right-hand side of the claimed inequality respectively. From the triangle inequality, $L$ is nonnegative, and clearly so is $R$. If one of the $z_k$'s is zero, then the inequality is immediate, so assume that no $z_k$ is zero.

Since $L$ is invariant under permutation of the points we can assume without loss of generality that $z_1,\dots,z_n$ are ordered by non-decreasing arguments, i.e., $z_k=|z_k|\mathrm{e}^{\mathrm{i}\phi_k}$ and $0\leq \phi_1 \leq \dots \leq \phi_n < 2\pi$. 
If two of the $z_k$ have the same argument then the inequality is trivial and strict since $L = 0 < R$.

So we assume that the arguments are pairwise distinct.
For $1 \leq k < n$ let $\gamma_k=\phi_{k+1}-\phi_k \in [0,2\pi)$ and $\gamma_n = 2\pi - (\phi_n-\phi_1)$ denote the angles between consecutive points. Then $\sum_{k=1}^n \gamma_k = 2\pi$, $0 < \gamma_k < 2\pi$ and $\Re(z_k \overline{z}_{k+1}) = |z_k||z_{k+1}|\cos \gamma_k$ for $1 \leq k \leq n$.
We write $D_k:=|z_k|+|z_{k+1}|$ and observe that
\[ 
|z_k+z_{k+1}|= \sqrt{|z_k|^2+|z_{k+1}|^2+2|z_k| |z_{k+1}| \cos \gamma_k} = \sqrt{D_k^2-2|z_k| |z_{k+1}| (1-\cos \gamma_k)}.
\]
By the AM-GM inequality we have $|z_k| |z_{k+1}| \leq \frac{D_k^2}{4}$ with equality if and only if $|z_k|=|z_{k+1}|$. From this we have that 
\begin{align}\label{eq: am-gm for regular n-gon}
|z_k+z_{k+1}|\geq \sqrt{D_k^2 \cdot \frac{1+\cos \gamma_k}{2}} = \sqrt{D_k^2 \cos^2 \frac{\gamma_k}{2} }\geq D_k \cos \frac{\gamma_k}{2},
\end{align}
where equality holds in both inequalities if and only if $|z_k|=|z_{k+1}|$ and $\gamma_k \leq \pi$.
Thus, for $1 \leq k \leq n$ we have $     |z_k|+|z_{k+1}|-|z_k+z_{k+1}|\le D_k\left(1-\cos\frac{\gamma_k}2\right)$ and so
\begin{align}\label{eq: zk zk+1 and Dk}
     L \leq D_k\left(1-\cos \frac{\gamma_k}{2}\right).
\end{align}
Let $f(x)=(1-\cos\frac{x}2)^{-1} = \frac{1}{2}\sin^{-2}\frac{x}{4}$ which is strictly convex on the interval $(0,2\pi)$. 
Since $1 - \cos \frac{\gamma_k}{2} > 0$ we obtain from Jensen's inequality and \eqref{eq: zk zk+1 and Dk} that
\begin{align}\label{eq: inequality with Jensen}
    nLf\left( \frac{2\pi}{n} \right)=nL f \left( \frac{\sum_{k=1}^n\gamma_k}{n}\right)\leq nL \sum_{k=1}^n \frac{f(\gamma_k)}{n} =\sum_{k=1}^n L f(\gamma_k)  \leq \sum_{k=1}^n D_k  = 2 \sum_{k=1}^n |z_k|
\end{align}
and thus $L \leq \frac{2}{n}\left( 1 - \cos \frac{\pi}{n} \right) \sum_{k=1}^n |z_k|=R$, which is what was to be shown.

Now we show that $L=R$ implies the stated equality condition.
If all $z_k$ are zero then clearly $L=R$. So, assume that at least one $z_k$ is nonzero. If one of the points is zero or if two of the points have the same argument then $L=0$ and $R > 0$ so the inequality is not tight. 
Suppose that all $z_1,\dots,z_n$ are nonzero and have pairwise distinct arguments. We assume again that the $z_1,\dots,z_n$ are ordered by non-decreasing arguments and follow the notation from the first part of the proof. 

If $L=R$ then \eqref{eq: inequality with Jensen} is tight and since $f$ is strictly convex we deduce that all $\gamma_k$ are equal, i.e., the arguments of all points $z_k$ are uniformly distributed.
Tightness in \eqref{eq: inequality with Jensen} together with \eqref{eq: zk zk+1 and Dk} implies that $Lf(\gamma_k)=D_k$ for all $k$ and hence throughout \eqref{eq: am-gm for regular n-gon}. Thus $|z_k|=|z_{k+1}|$ for all $1 \leq k \leq n$. So the points $z_1,\dots,z_n$ are indeed the vertices of an origin-centered regular $n$-gon.

Conversely, it is now evident that the vertices of an origin-centered regular $n$-gon satisfy $L=R$.
\end{proof}

\begin{corollary}\label{cor: alpha n2}
    For $n \geq 3$ we have $\alpha_{n,2}=\alpha_{n,n-2} = \frac{2}{n}(1-\cos \frac{\pi}{n})$.
\end{corollary}
\begin{proof}
From Lemma \ref{lemma: alpha n,d} we see that $\alpha_{n,2}=\alpha_{n,n-2}$, so it remains to prove that $\alpha_{n,2} = \frac{2}{n}(1-\cos \frac{\pi}{n})$. Let $P$ be a rank-$2$ orthogonal projection matrix; our goal is to relate $\min_{i < j} \lambda_{\min}(P_{\{i,j\}})$ to the left hand side of the inequality in Lemma \ref{lemma: min gap}.

Write $P=A^\top A$ for some $A \in \RR^{2 \times n}$ with $AA^\top =I_2$. 
For $1 \leq k \leq n$, let $(x_k,y_k)^\top$ denote the $k$-th column of $A$ and let $w_k := x_k + \mathrm{i}y_k$ and $z_k := w_k^2$. Then $|z_k| = x_k^2+y_k^2 = P_{kk}$ and so $\sum_{k=1}^n |z_k| = \tr (P)=2$. 
For $1 \leq j < k \leq n$ we have $\Re(w_j \overline{w}_k) = x_jx_k+y_jy_k$ and 
\[ P_{\{j,k\}} = \begin{bmatrix}
    |z_j| & x_jx_k+y_jy_k \\ x_jx_k+y_jy_k & |z_k|
\end{bmatrix}\]
and 
\[|z_j+z_k|^2 = |z_j|^2+|z_k|^2+2\Re (z_j \overline{z}_k)= (|z_j|+|z_k|)^2-4\det(P_{\{j,k\}}) =  \tr (P_{\{j,k\}})^2-4\det(P_{\{j,k\}}).\]
The eigenvalues of $P_{\{j,k\}}$ are 
\[ \frac{\tr(P_{\{j,k\}}) \pm \sqrt{\tr(P_{\{j,k\}})^2-4\det(P_{\{j,k\}})}}{2} = \frac{1}{2}(|z_j|+|z_k| \pm |z_j+z_k|)\] 
and the inequality in Lemma \ref{lemma: min gap} yields
\[ \min_{S \in \binom{[n]}{2}} \lambda_{\min}(P_S)  = \frac{1}{2 } \min_{j < k} (|z_j|+|z_k|-|z_j+z_k|) \leq \frac{1}{n} \left(1-\cos \frac{\pi}{n}\right) \sum_{k=1}^n |z_k| = \frac{2}{n}\left(1-\cos \frac{\pi}{n}\right)\]
which gives $\alpha_{n,2} \leq \frac{2}{n}(1-\cos \frac{\pi}{n}).$ 

To see that equality holds, consider the matrix $A \in \RR^{2 \times n}$ which $k$-th column is
\[\sqrt{\frac{2}{n}}\begin{bmatrix}
\cos \frac{(k-1) \pi}{n} \\
\sin \frac{(k-1)\pi}{n}
\end{bmatrix}.
\]
Then
\begin{align*} 
AA^\top &= \frac{2}{n} \sum_{k=1}^n \begin{bmatrix}
    \cos^2 \frac{(k-1)\pi}{n} & \cos  \frac{(k-1)\pi}{n} \sin  \frac{(k-1)\pi}{n} \\ \cos  \frac{(k-1)\pi}{n} \sin  \frac{(k-1)\pi}{n} & \sin^2  \frac{(k-1)\pi}{n}
\end{bmatrix}\\
&= \frac{1}{n}\sum_{k=1}^n\begin{bmatrix}
    1+\cos\frac{2(k-1)\pi}{n} & \sin \frac{2(k-1)\pi}{n} \\ \sin \frac{2(k-1)\pi}{n}  &  1-\cos\frac{2(k-1)\pi}{n}
\end{bmatrix} = I_2
\end{align*}
since $\sum_{k=1}^n \mathrm{e}^{\frac{2(k-1)\pi \mathrm{i}}{n}} =0$. 
Thus $P:=A^\top A$ is a rank-$2$ orthogonal projection matrix and the associated numbers $z_k = \frac{2}{n}\mathrm{e}^{\frac{2(k-1)\pi \mathrm{i}}{n}}$ are the vertices of an origin-centered regular $n$-gon. 
We conclude that equality holds in  \eqref{lemma: min gap} and hence $ \alpha_{n,2}=\frac{2}{n}(1-\cos \frac{\pi}{n})$ as claimed. 
\end{proof}

\begin{remark}\label{rem:3.5}
For $d \in \{2,n-2\}$ it follows directly from Lemma \ref{lem: alpha and epsilon} and Corollary \ref{cor: alpha n2} that the maximal $\varepsilon > 0$ such that $\left( \underbrace{1,\dots,1}_d,\underbrace{-\varepsilon,\dots,-\varepsilon}_{n-d} \right) \in \lambda(\S^{n,d})$ is
\[
\varepsilon = \frac{2\left(1-\cos \frac{\pi}{n} \right)}{n-2+2\cos\frac{\pi}{n}}.\]
\end{remark}

\begin{lemma}\label{lemma: decreasing}
For integers $1 \leq k < d < n$ we have $\alpha_{n,d} \leq \alpha_{n-k,d-k}$.
\end{lemma}
\begin{proof}
    Let $P=A^\top A$ with $A A^\top  = I_d$ and let $a_1,\dots,a_n \in \RR^d$ be the columns of $A$. 
    Choose $T \in \binom{[n]}{k}$. Since the dimension of the linear span of $\{a_i : i \in T\}$ is at most $k$ the dimension of its orthogonal complement in $\RR^d$ has dimension at least $d-k$. So there is a matrix $E \in \RR^{d \times (d-k)}$ with orthonormal columns such that $E^\top A_T=0$. 

    Let $W:= E^\top A_{T^c} \in \RR^{(d-k) \times (n-k)}$. Since $I_d= AA^\top = A_{T^c} A_{T^c}^\top + A_T A_T^\top $ we have
    \[ W W^\top = E^\top A_{T^c} A_{T^c}^\top E = E^\top (AA^\top - A_T A_T^\top)E = E^\top I_d E - (E^\top A_T)(E^\top A_T)^\top = I_{d-k} \]
    and so $Q:=W^\top W$ is a rank-$(d-k)$ orthogonal projection matrix in $\RR^{n-k}$. For a subset $R \in \binom{T^c}{d-k}$ let $S := T \cup R \in \binom{[n]}{d}$. Since $A_S$ and $W_R$ are square matrices the matrices $P_S=A_S^\top A_S$ and $A_S A_S^\top$ have the same eigenvalues, and $Q_R = W_R^\top W_R$ and $W_R W_R^\top$ have the same eigenvalues. Since $E^\top A_T=0$ and $W_R=E^\top A_R$ we have
    \[
    E^\top A_S A_S^\top E = E^\top (A_T 
    A_T^\top + A_R A_R^\top)E=E^\top A_R A_R^\top E=W_R W_R^\top. 
    \]
    Since $E$ has orthonormal columns, we obtain from the Rayleigh-Ritz theorem 
    \begin{align*}
    &\lambda_{\min}( P_S) = \lambda_{\min}(A_SA_S^\top) = \min_{x \in \RR^d : ||x||=1} x^\top A_S A_S^\top x \\
    &\leq \min_{x \in \RR^{d-k} : ||x||=1} x^\top E^\top A_S A_S^\top E x = \lambda_{\min}(W_R W_R^\top) = \lambda_{\min}(Q_R).
    \end{align*}
    By taking the minimum of the left hand side over all $R \in \binom{T^c}{d-k}$, we conclude
    \[ 
    \min_{S \in \binom{[n]}{d}} \lambda_{\min}(P_S) \leq \min_{R \in \binom{T^c}{d-k}} \lambda_{\min}(Q_R) \leq \alpha_{n-k,d-k}.
    \]
    Since the inequality holds independently of the choice of the rank $d$ orthogonal projection matrix $P$ we obtain $\alpha_{n,d} \leq \alpha_{n-k,d-k}$.
\end{proof}

\begin{proposition}\label{prop: bound}
    Let $P \in \Sym_n$ be a rank-$d$ orthogonal projection matrix. If $2\le d\le n-2$ then
    \[
    \min_{S\in\binom{[n]}d}\lambda_{\min}(P_S)\le\frac2{m+2}\left(1-\cos\frac{\pi}{m+2}\right)
    \]
    where $m=\max\{d,n-d\}$. The inequality is tight for $d\in\{2,n-2\}$.
\end{proposition}
\begin{proof}
For $d \geq 3$ setting $k:=d-2$ gives $\alpha_{n,d} \leq \alpha_{n-d+2,2} $ by Lemma \ref{lemma: decreasing}. 
For $d=2$ both sides trivially coincide. 
Moreover, Corollary \ref{cor: alpha n2} gives
$\alpha_{n-d+2,2} = \frac{2}{n-d+2}\left(1-\cos\frac{\pi}{n-d+2}\right)$. 
By Lemma \ref{lemma: alpha n,d} we have $\alpha_{n,d} = \alpha_{n,n-d}$ and thus
$
\alpha_{n,d} = \alpha_{n,n-d} \leq \alpha_{d+2,2} = \frac{2}{d+2}\left(1-\cos\frac{\pi}{d+2}\right).
$
Since the function $t \mapsto \frac{2}{t} \left(1-\cos\frac{\pi}{t}\right)=\frac{4}{t}\sin^2 \frac{\pi}{2t}$ is strictly decreasing on the interval $[2, \infty)$, we observe that
\[ \alpha_{n,d} \leq \min \left\{ \frac{2}{d+2} \left(1-\cos\frac{\pi}{d+2}\right), \frac{2}{n-d+2} \left(1-\cos\frac{\pi}{n-d+2}\right) \right\} = \frac{2}{m+2}\left(1-\cos\frac{\pi}{m+2}\right).\]
Finally, for $d \in \{2,n-2\}$ we have $m+2=n$ and the right hand side equals $\alpha_{n,2}=\alpha_{n,n-2}$ by Corollary \ref{cor: alpha n2}. 
Since $\alpha_{n,d}$ is attained at some orthogonal projection matrix, this shows that the inequality is tight for $d \in \{2,n-2\}$.
\end{proof}

\begin{theorem} \label{thm: nonconvexity all d}
Let $2 \leq d \leq n - 2$. The average
\[
\left(\underbrace{1, \ldots, 1}_{d}, \underbrace{-\frac1{d(n - d)}, \ldots,-\frac1{d(n - d)}}_{n-d} \right)
\]
of the $n - d$ points from $\lambda(\S^{n, d})$ in \eqref{eq:n-d points} is not contained in $\lambda(\S^{n, d})$. Consequently, $\lambda(\S^{n, d})$ is not convex.
\end{theorem}

\begin{proof}
    Let $m=\max\{d,n-d\}$. By Proposition \ref{prop: bound} we have 
    \[
    \alpha_{n,d}\le\frac2{m+2}\left(1-\cos\frac{\pi}{m+2}\right)<\frac{\pi^2}{(m+2)^3}<\frac1{m^2+1}\le\frac1{d(n-d)+1},
    \]
where the first strict inequality follows from $1- \cos x = 2\sin^2\frac{x}{2}< 2\frac{x^2}{4}=\frac{x^2}{2}$ for $x \geq 0$. 
The last inequality holds because $d(n-d) \leq m^2$ by assumption. 
Thus $\alpha_{n,d} < \frac{1}{d(n-d)+1} = \frac{\varepsilon'}{1+\varepsilon'}$ for $\varepsilon':= \frac{1}{d(n-d)}$. 
The average of the $n-d$ points in $\lambda(\S^{n,d})$ from the beginning of this section is $(1,\dots,1,-\varepsilon',\dots,-\varepsilon')$ with $d$ ones and $n-d$ copies of $-\varepsilon'$. It follows from Lemma \ref{lem: alpha and epsilon} that the average does not lie in $\lambda(\S^{n,d})$. In particular, the set $\lambda(\S^{n,d})$ is not convex.
\end{proof}
Since the hyperbolicity cones $H(e_d)$ are convex for $1 \leq d \leq n$, we observe that the average of the $n-d$ points in \eqref{eq:n-d points} lies in $H(e_d) \setminus \lambda(\S^{n,d})$ for all $2 \leq d \leq n-2$.

\section{Extremal points in \texorpdfstring{$\mathcal{L}_{n, d}$}{ln,d} and hyperplane bounds for \texorpdfstring{$\lambda(\mathcal{S}^{n,d})$}{lambda(Sn,d)}}\label{section: 2-locally psd-nonconvex}\label{sec:hypineq}

We investigate the sets $\mathcal{L}_{n,d}$ for $n \geq 4$ and $2 \leq d \leq n - 2$ through hyperplane bounds, i.e., linear inequalities that are valid on $\mathcal{L}_{n,d}$. Since the non-convexity of $\lambda(\S^{n,d})$ was established in Section \ref{sec:nonconv} we focus here on points at which such a bound is tight. More precisely, we inspect the points $(s,t)$ that minimize the sum $s+t$ subject to $(s,t) \in \mathcal{L}_{n,d}$. In the extreme cases $d \in \{2,n-2\}$ we show that this minimum cannot be attained when $s = t$. Moreover, for $d = 2$ we prove the stronger statement that this minimum can be achieved only when $s = 0$ or $t = 0$, i.e., the bound is tight at exactly two points. This allows us to describe in Theorem \ref{theorem: convex hull Ln2} the exact convex hull of $\mathcal{L}_{n,2}$. For $n = 4$, we already determined $\mathcal{L}_{4, 2}$ in Corollary \ref{cor: L42}.

We exhibit points on the boundary of $\mathcal{L}_{n,d}$ by finding the optimal values (in some instances) of two minimization problems.

\begin{problem}\label{problem:min} For $1 \leq d \leq n$ we define the following two optimization problems which ask for the minimum of a linear function over $\mathcal{L}_{n,d}$.
    \begin{align*}
        \mu_{n,d}^{\operatorname{tr}} & :=   \min\left\{s+t\,\suchthat\, (1,\dots,1,s,t)\in\lambda(\mathcal{S}^{n,d})\right\}, \\
        \pi_{n,d}^{\operatorname{tr}} & :=   \min\left\{t\,\suchthat\, (1,\dots,1,t,t)\in\lambda(\mathcal{S}^{n,d})\right\}.        
    \end{align*}
\end{problem}
Note that for $(n,d)=(4,2)$ we have $\alpha_{4,2}= -\frac{\pi_{4,2}^{\operatorname{tr}}}{1-\pi_{4,2}^{\operatorname{tr}}}$, where $\alpha_{n,d}$ is the value introduced in \eqref{eq:alpha nd}.
Below, in Proposition \ref{prop: equivalences}, we will see that the minima $\mu_{n,d}^{\operatorname{tr}}$ and $\pi_{n,d}^{\operatorname{tr}}$ are indeed attained. The following lemma shows that we can restrict our search to the third quadrant.

\begin{lemma}\label{lem: s,t <= 0}
Minimizing $s+t$ over $(s,t) \in \mathcal{L}_{n,d}$ with $s,t\leq 0$ yields the same value as minimizing $s+t$ over the whole $\mathcal{L}_{n,d}$. 
\end{lemma}

\begin{proof}
For $d=n$ this follows from the fact that $\mathcal{L}_{n,n} \subset \RR^2_{\geq 0}$. 
For $d=1$, \cite[Proposition~1.1]{kozhasov2023eigenvalues} tells us that $\lambda(\S^{n,1})=\left\{ \lambda \in \RR^n \, \suchthat \, \sum_{j=1}^n \lambda_j \geq 0 \right\}$, so $(s,t) \in \mathcal{L}_{n,1}$ if and only if $s + t \geq 2-n$. Therefore, $\mu_{n,1}^{\operatorname{tr}}$ is attained on the entire line through $(0,2-n)$ and $(2-n,0)$. 

We now assume $2 \leq d \leq n - 1$. Every
$(s,t) \in \mathcal{L}_{n,d}$ must satisfy $e_d(1,\dots,1,s,t) \ge 0$ since $\lambda(\mathcal{S}^{n,d}) \subseteq H(e_d)$.
If $s > 0$ and $t < 0$, then $st < 0$, and rearranging yields
\[
s + t \;\ge\; -\frac{n-d-1}{d} - \frac{d-1}{n-d}\, st
\;\ge\; -\frac{n-d-1}{d}\,;
\]
for $s,t \ge 0$ this bound holds trivially. The value
$-\frac{n-d-1}{d}$ is attained at
$\bigl(1,\dots,1,0,-\frac{n-d-1}{d}\bigr)$, which is in $\lambda(\S^{n, d})$ because it is the spectrum of $\frac{d-1}{d}\,G(n-1,d) \in \mathcal{S}^{n-1,d}$
extended by a zero entry. 
\end{proof}

A crucial step now is reformulating the optimization problems in Problem \ref{problem:min} as questions about optimal point configurations in the complex plane.

\begin{proposition}\label{prop: equivalences}
For $2 \leq d \leq n$, consider the following two optimization problems over $n$ complex numbers $z_1,\dots,z_n$:
    \begin{align*}
    \mu_{n,d}^{\max}  & :=   \min\left\{\max_{S\in\binom{[n]}d}\sum_{s\in S}|z_s|+\left|\sum_{s\in S} z_s\right|\,\suchthat\,z_1,\dots,z_n\in\CC,\, \sum_{k=1}^n|z_k|=1\right\}, \\
    \pi_{n,d}^{\max} & :=  
    \min\left\{\max_{S\in\binom{[n]}d}\sum_{s\in S}|z_s|+\left|\sum_{s\in S} z_s\right|\,\suchthat\,z_1,\dots,z_n\in\CC,\, \sum_{k=1}^n|z_k|=1,\, \sum_{k=1}^n z_k=0\right\}.    
    \end{align*} 
    The following relations hold between $\mu_{n,d}^{\tr}$ and $ \mu_{n,d}^{\max}$, as well as between $\pi_{n,d}^{\tr}$ and $ \pi_{n,d}^{\max}$: 
    \[ \mu_{n,d}^{\tr} = 2 \left( 1- \frac{1}{\mu_{n,d}^{\max}}\right) \text{ and } \pi_{n,d}^{\tr} = 1- \frac{1}{\pi_{n,d}^{\max}}.\]
    Moreover, for an optimal solution $(z_1,\dots,z_n)$ attaining $\mu_{n,d}^{\max}$ the corresponding pair $(s,t)$ of the optimization problem for $\mu_{n,d}^{\operatorname{tr}}$ satisfies $s=1-\frac{1}{\mu_{n,d}^{\max}}(1-|\sum_{j=1}^n z_j|)$ and $t=1-\frac{1}{\mu_{n,d}^{\max}}(1+|\sum_{j=1}^n z_j|)$, up to permutation of $s$ and $t$.
\end{proposition}
We will call a configuration $z_1, \ldots, z_n \in \CC$ \emph{optimal for $\mu_{n, d}^{\max}$} (resp.  $\pi_{n, d}^{\max}$) if it achieves the value of $\mu_{n, d}^{\max}$ (resp.  $\pi_{n, d}^{\max}$). Note that the optimization problems are continuous but non-smooth and non-convex. A proof of Proposition \ref{prop: equivalences} is presented in Appendix \ref{sec:appendixA}. Here we give a brief sketch.

\begin{proof}[Sketch of proof]
Let $M \in \Sym_n$ with $M \sim \diag (1,\ldots,1,s,t)$ and $\max\{s,t\} \leq 1$ which is no restriction by Lemma \ref{lem: s,t <= 0}. Since $I-M \sim \diag(0,\ldots,0,1-s,1-t) \succeq 0$, there exists a matrix $B \in \RR^{2 \times n}$ with $I-M=B^\top B$. Then $M \in \S^{n,d}$ if and only if $I_d-B_S^\top B_S \succeq 0$ for all $S \in \binom{[n]}{d}$.

Let $A \in \RR^{2 \times n}$ be a matrix proportional to $B$, and parametrize each of its columns by polar coordinates:
\[
A=\begin{bmatrix}
a_1\cos\b_1 & a_2\cos\b_2 & \dots & a_n\cos\b_n\\
a_1\sin\b_1 & a_2\sin\b_2 & \dots & a_n\sin\b_n
\end{bmatrix}.
\]
If $S \in \binom{[n]}{d}$, then after applying the Cauchy-Binet formula we find that the largest eigenvalue of $A_S^\top A_S$ is
\[
\frac12\left(\sum_{j\in S}|z_j|+\left|\sum_{j\in S} z_j\right|\right),
\] 
where $z_1,\dots,z_n$ are complex numbers such that $\Re z_j = a_j^2\cos 2\beta_j$ and $\Im z_j = a_j^2\sin 2\beta_j$. This gives us the desired relations.
\end{proof}

We note that the above ``complex squaring" trick has been independently used in \cite{nesterenko2024submatrices} to study 2-dimensional linear subspaces that maximally deviate from coordinate subspaces.

\begin{remark} \label{rem: strict implies nonconvex}
Note that for all $2 \leq d \leq n$ we have $\pi_{n,d}^{\max} \geq \mu_{n,d}^{\max}$, and a strict inequality implies that the hyperplane bound $s+t \geq \mu_{n,d}^{\operatorname{tr}}$ is not tight on the diagonal $s=t$. Assume $\pi_{n,d}^{\max} > \mu_{n,d}^{\max}$. By Proposition \ref{prop: equivalences}, this implies that any point $(s, t) \in \mathcal{L}_{n,d}$ at which $s + t$ is minimized satisfies $s \neq t$. 
Hence the bound is tight at the two distinct points $(s,t)$ and $(t,s)$.
In particular, this means that the midpoint of the line segment with endpoints $(s, t), (t, s) \in \mathcal{L}_{n,d}$ does not lie in $\mathcal{L}_{n,d}$ and $\mathcal{L}_{n,d}$ is not convex.
\end{remark}

\subsection{Conjectured optimal value of the min-max problem} 
We conjecture that the value of $\mu_{n,d}^{\max}$ is $\frac{2d}{n + d - 1}$, which can be formulated as the following inequality.
\begin{conjecture}\label{conjecture: optimal configurations}
Let $n\ge3$ and $1\le d\le n-1$. If $z_1,\dots,z_n\in\CC$, then
    \begin{align}\label{ineq: conj}
        \max_{S\in\binom{[n]}{d}}\sum_{s\in S}|z_s|+\left|\sum_{s\in S} z_s\right|\ge\frac{2d}{n+d-1}\sum_{k=1}^n|z_k|.
    \end{align}
\end{conjecture}

The conjecture holds for $d=1$, since the maximum number in a finite set of real numbers is always at least the average of the numbers in the set. We show that it holds for $d = n - 1$ and $d = 2$ in Proposition \ref{prop: d=n-1} and Theorem \ref{theorem: minmax d=2} respectively. We also verified that it holds for $d = n - 2$; however, to maintain the focus of this section, the lengthy proof is omitted from this paper and will appear as part of the fourth author’s Ph.D. thesis \cite{lee2026}. The conjecture fails for $d=n$. 

We observe that the conjecture always holds for the trivial configuration $z_1 = \ldots = z_n = 0$, so for the remainder of this paper, we will implicitly exclude this configuration and assume $\sum_{k=1}^n |z_k| > 0$. Since the inequality is homogeneous in $z_1, \ldots, z_n$, we may assume without loss of generality that $\sum_{k=1}^n |z_k| = n + d - 1$ and we will do so in our discussions unless specified otherwise.

For all $d\ge2$, equality is attained in Conjecture \ref{conjecture: optimal configurations} for the one-parameter family of configurations
\begin{align}\label{eq: configs}
    (z_1, \ldots, z_n) \in \left\{ (\underbrace{1,\dots,1}_{n-2},u,1-d-u) \suchthat u\in\CC \,\,\text{on the ellipse}\,\, |z|+|1-d-z|=d+1\right\}.
\end{align}
Observe that $u=1$ is always feasible, which gives us the configuration $(1,\dots,1, 1,-d)$. 

We show in Theorem \ref{theorem: optimal config d=2} that for $n \geq 4$ and $d=2$, these configurations are precisely the optimal ones for $\mu_{n,d}^{\max}$ (up to permutation and complex rescaling). We also verified that this is the case for $n \geq 4$ and $d = n - 2$, but again we omit the proof here and defer it to \cite{lee2026}.
In Corollary \ref{cor: optimality of regular n-gon} we establish that for $n \geq 4$ and $d=n-2$ strict inequality occurs in Conjecture \ref{conjecture: optimal configurations} when restricting to point configurations where $\sum_{k=1}^n z_k = 0$.

\begin{remark} \label{rem: sharp?}
The existence of the above configurations implies $\mu_{n, d}^{\max} \leq \frac{2d}{n + d - 1}$ for all $2 \leq d \leq n - 2$. 
\end{remark}

\begin{proposition}\label{prop: d=n-1}
    Let $n \geq 3$. For $z_1,\dots,z_n\in\CC$ we have
    \begin{align*}
        \max_{S\in\binom{[n]}{n-1}}\sum_{s\in S}|z_s|+\left|\sum_{s\in S} z_s\right|\ge\sum_{k=1}^n|z_k|
    \end{align*}
    with equality if and only if $\sum_{k=1}^n z_k=0$.
\end{proposition}
\begin{proof}
    Let $T=\sum_{k=1}^n z_k$. The given inequality rearranges to
    \begin{align*}
        \max_{k=1,\dots,n}|T-z_k|-|z_k|\ge0.
    \end{align*}
    For each $k \in [n]$, let $z_k=x_k+\mathrm{i}y_k$, where $x_k,y_k\in\RR$. Suppose that $|z_k|>|T-z_k|$ for all $k \in [n]$. Squaring each of these inequalities and adding them up, we obtain
    \begin{align*}
        \sum_{k=1}^n (x_k^2+y_k^2)&>\sum_{k=1}^n(\Re(T)-x_k)^2+(\Im(T)-y_k)^2\\
       \Rightarrow 0 &>(n-2)(\Re(T)^2+\Im(T)^2),      
    \end{align*}
    a contradiction. Therefore, we have $|z_k| \leq |T - z_k|$ for some $k$, which gives the inequality. Additionally, equality is attained if and only if $|z_k| \geq |T - z_k|$ for all $k$, where at least one $k$ attains equality. By the same process of squaring and adding these inequalities, we see that this forces $\Re(T)=0$ and $\Im(T)=0$, i.e., $T=0$. Conversely, if $T = 0$ then the left hand side is zero and we attain equality.
\end{proof}
In particular, Proposition \ref{prop: d=n-1} implies that the values $\mu_{n,n-1}^{\max}$ and $\pi_{n,n-1}^{\max}$ are equal. 

\subsection{The hyperplane bound for \texorpdfstring{$\mathcal{L}_{n,2}$}{Ln2}}
\label{sec:Conjectured optimal value}
The following theorem establishes Conjecture \ref{conjecture: optimal configurations} for $d=2$. The ideas of the proof allow us to find all optimal configurations and to show that the resulting hyperplane bound for $\mathcal{L}_{n,2}$ is tight at exactly two points, for all $n \geq 4$.

\begin{theorem}\label{theorem: minmax d=2}
Let $n \geq 3$. If $z_1,\dots,z_n\in\CC$,  then
    \[
        \max_{1\le k<l\le n}|z_k|+|z_l|+|z_k+z_l|\ge\frac4{n+1}\sum_{k=1}^n|z_k|~.
    \]
Equivalently, we have $s+t \geq \frac{3-n}{2}$ for all $(s,t) \in \mathcal{L}_{n,2}$ and this bound is tight.
\end{theorem}
Recall that the inequality is tight for configurations $(1,\ldots,1,u,-1-u)$ with $u \in \CC$ on the ellipse $|z|+|1+z|=3$. Thus, the equivalence of the two assertions in the theorem follows from Proposition \ref{prop: equivalences}. The proof relies on the validation of several trigonometric inequalities. We only give an outline here and the full proof can be found in Appendix \ref{sec:appendixC}. 
\begin{proof}[Sketch of proof]
    Let $S:=\sum_{k=1}^n|z_k|$.
    If any $|z_m|\ge\frac2{n+1}S$ then by the triangle inequality we have $|z_m|+|z_k|+|z_m+z_k|\ge2|z_m|\ge\frac4{n+1}S$.
    If any $|z_m|\le\frac1{n+1}S$ then $S-|z_m|\ge\frac{n}{n+1}S$, and by induction on $n$ (base case is Proposition \ref{prop: d=n-1} with $n=3$),
    \[
    \max_{k<l\in[n]\setminus\{m\}} |z_k|+|z_l|+|z_k+z_l|\ge\frac4{n}(S-|z_m|)\ge\frac4{n}\cdot\frac{n}{n+1}S=\frac4{n+1}S~.
    \]
    So from now on suppose $\frac1{n+1}S<|z_k|<\frac2{n+1}S$ for $k=1,\dots,n$.
    Suppose $z_1,\dots,z_n$ are ordered increasingly by argument, and let $r_k:=|z_k|$.
    Let $\gamma_k$ be the angle between $z_k$ and $z_{k+1}$ for $k=1,\dots,n$. So $\gamma_k\in[0,\pi]$ and $2\pi\ge\sum_{k=1}^n\gamma_k$. 
    Let $C:=\frac4{n+1}S$ and suppose that $|z_k|+|z_{k+1}|+|z_k+z_{k+1}|<C$ for $k=1,\dots,n$. Hence,
    \begin{align*}
        r_k+r_{k+1}+\sqrt{r_k^2+r_{k+1}^2+2r_kr_{k+1}\cos\gamma_k} &< C\\
        \Rightarrow r_k^2+r_{k+1}^2+2r_kr_{k+1}\cos\gamma_k &< (C-r_k-r_{k+1})^2\\
        \Rightarrow \cos\gamma_k &< \frac{(C-r_k-r_{k+1})^2-(r_k^2+r_{k+1}^2)}{2r_k r_{k+1}}\\
        \Rightarrow \gamma_k &> \arccos\left(\frac{(C-r_k-r_{k+1})^2-(r_k^2+r_{k+1}^2)}{2r_k r_{k+1}}\right)~.
    \end{align*}

    Let
    \[
    F(r_1,\dots,r_n):=\sum_{k=1}^n \arccos\left(\frac{(C-r_k-r_{k+1})^2-(r_k^2+r_{k+1}^2)}{2r_k r_{k+1}}\right)~.
    \]    
    By Proposition \ref{prop: spherical inequality}, we have $F\ge2\pi$ on the simplex $\{x\in\RR^n\, \suchthat \, |x|_1=S, \frac14C\le x_i\le \frac12C\}$, so
    \[
    2\pi\ge\sum_{k=1}^n\gamma_k>F(r_1,\dots,r_n)\ge2\pi
    \]
    a contradiction.
\end{proof}

Following the proof of Theorem \ref{theorem: minmax d=2}, we recover all optimal point configurations.

\begin{remark}\label{remark: consecutive angles}
    Using the conventions and notation of the proof of Theorem \ref{theorem: minmax d=2}, observe that any configuration $z_1,\dots,z_n\in\CC$ that is optimal for $\mu_{n, 2}^{\max}$ must satisfy $\sum_{k=1}^n \gamma_k=2\pi$. Otherwise, $z_1,\dots,z_n$ are contained in a sector of angle $\gamma<\pi$ and it is possible to increase all the angles between consecutive $z_k$ in the sector, while remaining in the sector, to obtain a new point configuration where $\max_{1\le k<l\le n}|z_k|+|z_l|+|z_k+z_l|<\frac{4}{n+1}\sum_{k=1}^n|z_k|$, a contradiction.
\end{remark}
The following Proposition yields a recursive principle (Lemma \ref{lem: recursive principle}) for the classification of optimal point configurations for $\mu_{n,2}^{\max}$. This allows us to give a proof of the classification via induction.
\begin{proposition}\label{prop: two of length one}
    Let $n \geq 4$. If a point configuration $z_1,\ldots,z_n\in\CC$ is optimal for $\mu_{n,2}^{\max}$, then there exist distinct $m,m'\in [n]$ such that $z_m=z_{m'}$ and $|z_m|=|z_{m'}|=\frac1{n+1}\sum_{k=1}^n |z_k|$.
\end{proposition}
The proof requires lemmas also needed to prove Theorem \ref{theorem: minmax d=2} and can be found in Appendix \ref{sec:appendixC}.

\begin{lemma}\label{lem: recursive principle}
Let $n \geq 4$ and assume $z_1, \ldots, z_n \in \CC$ is an optimal configuration for $\mu_{n, 2}^{\max}$ such that $|z_1| = \frac{1}{n + 1}\sum_{k=1}^n |z_k|$. Then $z_2, \ldots, z_n$ is an optimal configuration for $\mu_{n - 1, 2}^{\max}$.
\end{lemma}

\begin{proof}
By Theorem \ref{theorem: minmax d=2} and the assumption, we have
\begin{align*}
4=\frac{4}{n+1}\sum_{k=1}^n|z_k|&=\max_{1 \leq k < l \leq n}|z_k|+|z_l|+|z_k+z_l| \nonumber \\ 
&\geq \max_{2 \leq k < l \leq n}|z_k|+|z_l|+|z_k+z_l| \geq \frac{4}{n}\sum_{k=2}^n |z_k|=4 
\end{align*} 
so $\max_{2 \leq k < l \leq n}|z_k|+|z_l|+|z_k+z_l| = 4$ and hence $z_2, \ldots, z_n$ is an optimal configuration for $\mu_{n-1, 2}^{\max}$.
\end{proof}

\begin{theorem}\label{theorem: optimal config d=2}
    Let $n \geq 4$. A point configuration $z_1,\ldots,z_n \in \CC$ is optimal for $\mu_{n,2}^{\max}$ if and only if 
    \begin{center}
        $z_1=\cdots = z_{n-2}$,\quad $|z_1|=\frac{1}{n+1}\sum_{k=1}^n|z_k|,$\quad and\quad $z_1+z_{n-1}+z_n=0$
    \end{center}  
    holds up to permutation of the complex numbers $z_1,\ldots,z_n$.
\end{theorem}
\begin{proof} We first see that the given configuration is indeed optimal; if $z_1=\cdots = z_{n-2}$, $|z_1| =1$, and $z_1+z_{n-1}+z_n=0$, then we find 
\[
\max_{1 \leq k < l \leq n-2}|z_k|+|z_l|+|z_k+z_l| = 4
\]
and 
\[
\max_{k<l \in \{1,n-1,n\}}|z_k|+|z_l|+|z_k+z_l|=|z_1|+|z_{n-1}|+|z_n|=n+1-(n-2)+1=4.
\]
It follows from Theorem \ref{theorem: minmax d=2} that $z_1,\ldots,z_n$ is indeed an optimal configuration.

Now assume that $z_1, \ldots, z_n \in \CC$  is an optimal configuration. We will show by induction on $n$ that it must have the described form. 
The base case is when $n = 4$. Let $z_1,z_2,z_3,z_4 \in \CC$ be an optimal configuration for $\mu_{4, 2}^{\max}$. By Proposition \ref{prop: two of length one}, there are distinct indices $m, m' \in [4]$ such that $z_m = z_{m'}$ and $|z_m| = |z_{m'}| = 1$, where we  assume without loss of generality that $m = 1$, $m' = 2$, and $z_1 = 1$.
Note that $|z_3|+|z_4|=3$ and $\max_{1 \leq j < k \leq 4} |z_j| + |z_k| + |z_j + z_k| = 4$, so
\[
|z_1| + |z_3| + |z_1 + z_3| \leq 4 \quad \Rightarrow \quad |1 + z_3| \leq 4 - |z_1| - |z_3| = |z_4|.
\]
Squaring both sides yields $1 + |z_3|^2 + 2\Re z_3 \leq |z_4|^2$ and likewise $1 + |z_4|^2 + 2\Re z_4 \leq |z_3|^2$. Adding these gives \[
2 + 2\Re (z_3 + z_4) \leq 0 \quad \Rightarrow \quad \Re (z_3 + z_4) \leq -1.
\]
On the other hand, we have $
|z_3| + |z_4| + |z_3 + z_4| \leq 4$, so $|z_3 + z_4| \leq 1.
$
Together, these force $z_3 + z_4 = -1$. Thus, the optimal configuration has the described form.

For the inductive step, assume that every optimal configuration has the described form for $n$. Let $z_1,\ldots,z_{n+1}$ be optimal for $\mu_{n+1,2}^{\max}$. Again by Proposition \ref{prop: two of length one}, we may assume without loss of generality that $z_1 = z_2 = 1$. By Lemma \ref{lem: recursive principle}, we also see that the configuration $z_2,\ldots,z_{n+1}$ is optimal for $\mu_{n, 2}^{\max}$. It follows from the induction hypothesis that there are $n - 2$ numbers from $z_2, \ldots, z_{n + 1}$ all equal and with modulus 1. Suppose that $z_2$ is not one of those numbers; by the induction hypothesis, there is $b \in \{3, \ldots, n + 1\}$ and $w \in \CC$ such that $|w| = 1$, $w \neq 1$, $w + 1 + z_b = 0$, and $|z_b| = 2$. This forces $|w + 1| = 2$ and hence $w = 1$, which is a contradiction. So, $z_2$ must be one of those $n - 2$ numbers, say $z_2 = \cdots = z_{n-1} = 1$ with $z_n+z_{n+1}=-1$ and $|z_{n}|+|z_{n+1}|=3$. So we conclude that the configuration $z_1, \ldots, z_{n + 1}$ has the described form and that the statement is true for $n + 1$.
\end{proof}

The following corollary is a direct consequence of Theorems \ref{theorem: minmax d=2} and \ref{theorem: optimal config d=2}.

\begin{corollary}\label{corollary: nonconvexity}
    For $n \geq 4$, the only optimal solutions for $\mu_{n,2}^{\operatorname{tr}}$ are $(s, t)=(0, \frac{3-n}{2})$ and $(s, t)=(\frac{3-n}{2},0)$, i.e., the hyperplane bound $s+t \geq \frac{3-n}{2}$ from Theorem \ref{theorem: minmax d=2} is tight at exactly these two points of $\mathcal{L}_{n,2}$. In particular, the relative interior of the line segment between these two points is disjoint from $\mathcal{L}_{n,2}$.
\end{corollary}
\begin{proof}
By Proposition \ref{prop: equivalences} for an optimal solution $(s,t)$ we have $s=1-\frac{n+1}{4}(1-|\sum_{j=1}^n z_j|)$ and $t=1-\frac{n+1}{4}(1+|\sum_{j=1}^n z_j|)$,
up to permutation, since $\mu_{n,2}^{\max}=\frac{4}{n+1}$ by Theorem \ref{theorem: minmax d=2}. 
To explicitly calculate the value of $|\sum_{j=1}^n z_j |$ we use the fact that every optimal configuration satisfies $z_1= \dots = z_{n-2}$, $|z_1|=\frac{1}{n+1}$ and $z_1+z_{n-1}+z_n=0$ by Theorem \ref{theorem: optimal config d=2}. From this, we have $\sum_{j=1}^n z_j = (n-3)z_1$ and $| \sum_{j=1}^n z_j | = \frac{n-3}{n+1}$. We conclude that $s=0$ and $t=\frac{3-n}{2}$ holds up to permutation. 
In particular, the relative interior of the line segment between $(0,\frac{3-n}{2})$ and $(\frac{3-n}{2},0)$ is not contained in $\mathcal{L}_{n, 2}$. 
\end{proof}

Using Proposition \ref{prop: 2nd 4th quadrant}, we now describe the convex hull of $\mathcal{L}_{n,2}$ and its boundary curves in the second and fourth quadrants.

\begin{proposition}\label{prop: 2nd 4th quadrant}
Let $n \geq 4$. All points $(s,t)$ in the second and fourth quadrants satisfying $e_2(1,\dots,1,s,t)=0$ belong to the boundary of $\mathcal{L}_{n,2}$.
\end{proposition}
\begin{proof}
In fact, every point on this boundary can be obtained by diagonal scaling from the matrix $G(n, 2) \in \S^{n, 2}$ (recall \eqref{eq: G(n,d)}).
Assume that 
\[
e_2(1,\dots,1,s,t) = \binom{n-2}2+(n-2)(s+t)+st=0
\]
and let 
\[M= DG(n, 2)D = D(2I - \mathbf{1}\mathbf{1}^\top)D = 2D^2-vv^\top,\] 
where 
\[
x=\frac{-st}{n-2}, \quad D=\diag\left(\frac{1}{\sqrt{2}},\ldots, \frac{1}{\sqrt{2}}, \sqrt{x}\right), \quad v=D\mathbf{1}=\left(\frac{1}{\sqrt{2}}, \ldots, \frac{1}{\sqrt{2}}, \sqrt{x}\right).
\]

Computing $\det(\lambda I-M)$ using the matrix determinant lemma, we have 
\begin{align*}
    \det(\lambda I-M)&=\det(\lambda I-2D^2)(1+v^\top(\lambda I-2D^2)^{-1}v)\\
    &=(\lambda-1)^{n-1}(\lambda-2x)\left(1+\frac{n-1}{2(\lambda-1)}+\frac{x}{\lambda-2x}\right)\\
    &=(\lambda-1)^{n-2}(\lambda-s)(\lambda-t)~.
\end{align*}
\end{proof}

\begin{theorem}\label{theorem: convex hull Ln2}
    For $n \geq 4$, the convex hull of the set $\mathcal{L}_{n,2}$ is given by the inequalities 
    \begin{align*}
        e_2(1,\dots,1,s,t)\ge0, \quad e_1(1,\dots,1,s,t)\ge\frac{n-1}2.
    \end{align*}
\end{theorem}
\begin{proof} 
    The first quadrant is clearly contained in $\mathcal{L}_{n,2}$ and the set $\mathcal{L}_{n,2}$ is star-shaped with respect to $\RR_{\ge 0}^2$.
    Any $(s,t) \in \mathcal{L}_{n,2}$ must satisfy $e_2(1,\ldots,1,s,t) \geq 0$, and by Proposition \ref{prop: 2nd 4th quadrant}, the points in the second and fourth quadrants with $e_2(1,\dots,1,s,t)=0$ are in $\mathcal{L}_{n,2}$. Since the set of all $(s,t)$ such that $e_1(1,\ldots,1,s,t) \geq 0$ and $e_2(1,\ldots,1,s,t) \geq 0$ is convex, its intersection with each quadrant is also convex. Together with star-shapedness, this shows that all points in the second and fourth quadrants satisfying these inequalities lie in $\mathcal{L}_{n,2}$. Finally, Theorem \ref{theorem: minmax d=2} gives $s+t\ge\frac{3-n}2$, so this is the tightest linear inequality in the third quadrant. 
\end{proof}

\subsection{Regular polygons and \texorpdfstring{$\mathcal{L}_{n,n-2}$}{Ln,n-2}}
For $d=n-2$ the second optimization problem of Proposition \ref{prop: equivalences} is solved by Lemma \ref{lemma: min gap}. More specifically, the optimal configurations attaining $\pi_{n,n-2}^{\max}$ are the vertices of origin-centered regular $n$-gons.

\begin{corollary}\label{cor: optimality of regular n-gon}
Let $n \geq 4$. Then
\[ \pi_{n,n-2}^{\max} = 1-\alpha_{n,2} = \frac{n-2}{n} +\frac{2}{n} \cos \frac{\pi}{n} \]
and the optimal configurations for $\pi_{n,n-2}^{\max}$ are precisely the vertices of origin-centered regular $n$-gons with $\sum_{k=1}^n |z_k|=1$. Moreover, $\pi_{n,n-2}^{\max} > \mu_{n,n-2}^{\max}$.
\end{corollary}
\begin{proof}
Let $z_1,\dots,z_n \in \mathbb{C}$ with $\sum_{k=1}^n |z_k| =1 $ and $\sum_{k=1}^n z_k = 0$. For $S = [n] \setminus \{i,j\}$ we have $\sum_{s \in S} z_s = -(z_i+z_j)$ and thus
\[ \sum_{s \in S} |z_s| + \left| \sum_{s \in S} z_s \right| = 1- (|z_i|+|z_j|-|z_i+z_j|).\]
So by Lemma \ref{lemma: min gap} we have
\[ \max_{S \in \binom{[n]}{n-2}} \sum_{s \in S} |z_s| + \left| \sum_{s \in S} z_s \right| = 1 - \min_{i < j} (|z_i|+|z_j|-|z_i+z_j|) \geq 1- \frac{2}{n}\left( 1 - \cos \frac{\pi}{n} \right).\]
Moreover, equality holds if and only if $z_1,\dots,z_n$ are the vertices of an origin-centered regular $n$-gon.
Such a configuration is feasible, since the vertices sum to zero. Together with $\alpha_{n,2} = \frac{2}{n}\left( 1 - \cos \frac{\pi}{n} \right)$ from Corollary \ref{cor: alpha n2} this shows the first claim. 

Finally, Remark \ref{rem: sharp?} gives $\mu_{n,n-2}^{\max} \leq \frac{2n-4}{2n-3} = 1 - \frac{1}{2n-3}$. Since $\alpha_{n,2} < \frac{1}{2n-3}$ we obtain $\pi_{n,n-2}^{\max} = 1- \alpha_{n,2} > \mu_{n,n-2}^{\max}$.
\end{proof}

\begin{remark}
By Corollary \ref{cor: optimality of regular n-gon} and Remark \ref{rem: strict implies nonconvex} the hyperplane bound $s+t \geq \mu_{n,n-2}^{\operatorname{tr}}$ for $\mathcal{L}_{n,n-2}$ is not tight on the diagonal $s=t$. It is at least tight at the two distinct points $(s,t)$ and $(t,s)$ whose midpoint does not lie in $\mathcal{L}_{n,n-2}$. Moreover, Proposition \ref{prop: equivalences} implies $\pi_{n,n-2}^{\operatorname{tr}} = - \frac{\alpha_{n,2}}{1-\alpha_{n,2}}$ which agrees with Remark \ref{rem:3.5}. Since Conjecture \ref{conjecture: optimal configurations} holds for $d=n-2$ and the configurations \eqref{eq: configs} are the only optimal ones \cite{lee2026} one obtains, as for $d=2$, a stronger statement: we have $\mu_{n,n-2}^{\operatorname{tr}}= - \frac{1}{n-2}$ and the only points of $\mathcal{L}_{n,n-2}$ at which the bound $s+t \geq -\frac{1}{n-2}$ is tight are $(0,-\frac{1}{n-2})$ and $(-\frac{1}{n-2},0)$.
\end{remark}

\section{Conjectures and Open Problems}

We described non-convexity of $\mathcal{L}_{n, 2}$ and $\mathcal{L}_{n,n- 2}$ in the third quadrant for $n \geq 4$, and we completely understand the case $n=4$. For $n \geq 5$ we do not know a full semialgebraic description of these sets. 

A particular point of interest is when $s = t$ on the boundary of $\mathcal{L}_{n, 2}$. As discussed in Proposition \ref{prop: equivalences}, this is equivalent to finding $\pi_{n, 2}^{\max}$ (recall that we established $\pi_{n,n-2}^{\max}$ in Corollary \ref{cor: optimality of regular n-gon}).
We reformulate $\pi_{n, 2}^{\max}$ as an optimization problem of convex $n$-gons in the plane analogously to Nesterenko \cite{nesterenko2024submatrices}.
These extremal problems structurally relate to the problem of finding small polygons (with bounded diameter) of maximal area or perimeter \cite{Henrion13,Bingane24,Mulansky2025}.  

Let $z_k=r_k\mathrm{e}^{\mathrm{i}\theta_k}$ for $k=1,\dots,n$. We may let $0\le\theta_1\le\theta_2\le\dots\le\theta_n<2\pi$ without loss of generality. Since $\sum_{k=1}^n z_k=0$ we can think of $z_1,\dots,z_n$ as vectors $u_1,\dots,u_n\in\RR^2$ corresponding to consecutive sides of a convex $n$-gon with vertices $A_1,A_2,\dots, A_n$ with $A_1$ at the origin and $u_1=\overrightarrow{A_1A_2},\, u_2=\overrightarrow{A_2A_3},\dots,\, u_{n-1}=\overrightarrow{A_{n-1}A_n},\, u_n=\overrightarrow{A_nA_1}$. In this way, $|u_k|+|u_{k+1}|+|u_k+u_{k+1}|$ is the perimeter of the triangle $A_kA_{k+1}A_{k+2}$ for $k=1,\dots,n$. 

For a polygon $\X$, denote its perimeter by $p(\X)$. We write $\X=X_1\cdots X_n$ to mean that $\X$ has vertices $X_1,\dots,X_n$. Let 
\[
\pi(\X):= \max_{1\le k<l\le n}|\overrightarrow{X_k X_{k+1}}|+|\overrightarrow{X_l X_{l+1}}|+|\overrightarrow{X_k X_{k+1}}+\overrightarrow{X_l X_{l+1}}|~.
\]
For a nondegenerate polygon $\X$, let
\begin{align*}
    \d(\X):=p(\X)^{-1}\pi(\X)~.
\end{align*}

\begin{problem}\label{problem: polygon}
    Let $n\ge5$. Find $\pi_{n,2}^{\max}=\min\{\d(\X)\,\suchthat\, \X \text{ non-trivial convex } n\text{-gon}\}$.
\end{problem}

Observe that the minimum in Problem \ref{problem: polygon} is actually achieved by a polygon, due to compactness of the projective space. We call a convex  $n$-gon that achieves this minimum \emph{optimal}.

\begin{proposition}
    If $\A_n$ is a regular $n$-gon, then $\delta(\A_n)=\frac4n\cos^2\frac{\pi}{2n}$. Consequently, we have $\pi_{n,2}^{\max} \leq \frac4n\cos^2\frac{\pi}{2n}$.
\end{proposition}
\begin{proof}
    If $\A_n=A_1\cdots A_n$, then we have $\pi(\A_n)=p(A_1A_2A_3)$. Setting $p(\A_n)=1$, observe that $p(A_1A_2A_3)=\frac4n\cos^2\frac{\pi}{2n}$.
\end{proof}

Note that for $n = 4$, this bound is tight; Corollary \ref{cor: L42} and Proposition \ref{prop: equivalences} imply that $\pi_{4, 2}^{\max}=\frac{2 + \sqrt2}{4}$, which is achieved when $\mathcal{X}$ is a square. Alternatively, this follows for $n=4$ from Corollary \ref{cor: optimality of regular n-gon}.

We conjecture that for $n=5$ the optimal polygon is approximately $1.3\%$ better than the regular pentagon.

\begin{conjecture} \emph{The house}, corresponding to a convex pentagon $ABCDE$ where $ABCE$ is a rectangle and $|AB|=30$, $|BC|=|EA|=16$, $|CD|=|DE|=25$ (see Figure \ref{fig:house1}), is an optimal polygon with $\pi_{5, 2}^{\max}=\frac57$. We have verified that the house is a strict local minimum up to rotation and scaling.
\end{conjecture}

\begin{figure}[h!]
    \centering
    \begin{subfigure}[b]{0.48\textwidth}
        \centering
        \includegraphics[width=\textwidth]{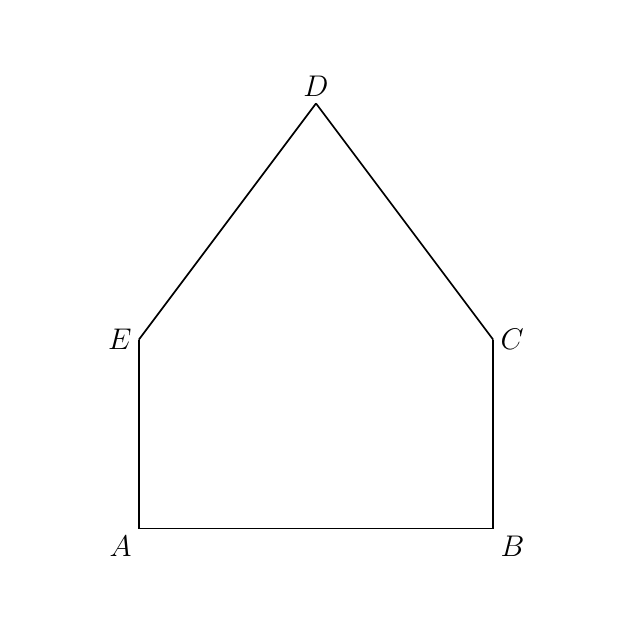}
        \caption{\emph{The house}. Lattice polygon, sides and diagonals are of integer length.}
        \label{fig:house1}
    \end{subfigure}    
    \label{fig:house}
    \hfill
    \begin{subfigure}[b]{0.48\textwidth}
        \centering
        \includegraphics[width=\textwidth]{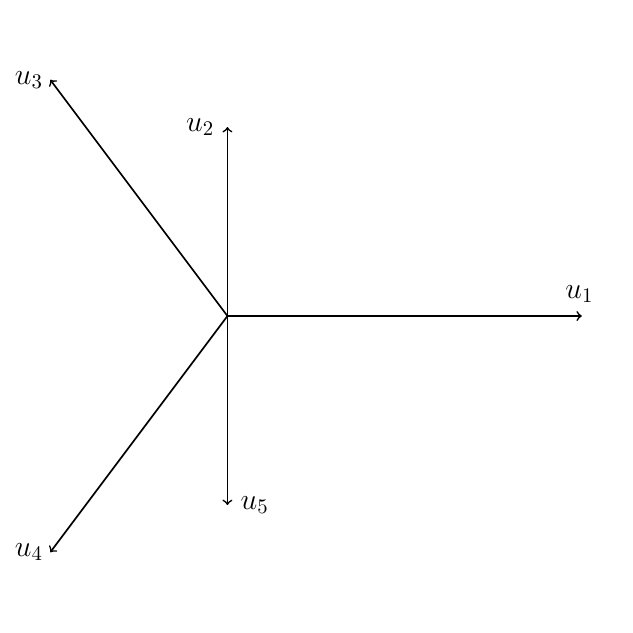}
        \caption{Vector configuration of \emph{the house}. Here $u_1=\overrightarrow{AB}$, $u_2=\overrightarrow{BC}$, $u_3=\overrightarrow{CD}$, $u_4=\overrightarrow{DE}$, $u_5=\overrightarrow{EA}$.}
        \label{fig:house2}
    \end{subfigure}
\end{figure}

Numerical experiments for $d = 2$ and $n \geq 5$ converge toward a particular pattern of polygons that improve the quantity of $\delta(\mathcal{X})$ over the regular $n$-gon configuration. For even $n$, we formulate the following conjecture.

\begin{conjecture} \label{conj: n>=6 even}
    For $n \geq 6$ even, the value of $\pi_{n,2}^{\max}$ is achieved by the following \emph{3-4 rectangle} configuration. Take $\frac{n-2}2$ numbers equal to $3$, $\frac{n-2}2$ equal to $-3$, one equal to $4\mathrm{i}$ and one equal to $-4\mathrm{i}$. Equivalently,
    \[
    \max_{1\le i<j\le n}|z_i|+|z_j|+|z_i+z_j|\ge\frac{4}{n+\frac23}\sum_{k=1}^{n}|z_k|
    \]
    for any complex numbers $z_1,\dots,z_{n}$ that add up to zero, where equality is attained by 
    \[z_1 = \cdots = z_{\frac{n}2-1} = 3, \quad z_{\frac{n}2}=4\mathrm{i}, \quad z_{\frac{n}2 + 1} = \cdots = z_{n-1} = -3, \quad z_n=-4\mathrm{i}.
    \]
\end{conjecture}

Since $\frac4n\cos^2\frac{\pi}{2n}>\frac4{n+\frac23}$ the above \emph{3-4 rectangle} configuration shows that when $d=2$ and $n\ge6$ is even, the optimal polygon is never the regular $n$-gon; this is in contrast to when $d=n-2$, where the optimal polygon is precisely the regular $n$-gon due to Corollary \ref{cor: optimality of regular n-gon}. 

When $d=2$ and $n\ge7$ is odd, our experiments suggest that regular polygons are also suboptimal and that $\pi_{n,2}^{\max}$ is asymptotic to $\frac{4}{n+\frac59}$.

\begin{remark}
The conjecturally optimal configurations above exhibit a phenomenon familiar from min-max extremal problems (e.g. maximal-perimeter small polygons of maximal area in
\cite{Henrion13,Bingane24}): the maximum is attained
simultaneously by many pairs. \emph{The house} attains it for $7$ of the
$\binom{5}{2}=10$ pairs, and the \emph{3-4 rectangle} for
$m^2+m-2$ of the $2m^2-m$ pairs, where $n=2m$; this is roughly
half of all pairs. Without the constraint $\sum_k z_k=0$, all the optimal configurations of
Theorem~\ref{theorem: optimal config d=2}, and conjectured ones in \eqref{eq: configs}, \emph{all} $\binom{n}{d}$ pairs
attain the maximum.
Determining the maximal number of simultaneously active
pairs for a non-trivial configuration with $\sum_k z_k = 0$, and
whether the optimal configurations for $\pi^{\max}_{n,2}$ always
achieve this maximum is an interesting open question. 
\end{remark}

\appendix

\section{Proof of Proposition \ref{prop: equivalences}}\label{sec:appendixA}

We present a proof of Proposition \ref{prop: equivalences} which relates the optimal configuration problem over $n$ complex numbers to the problem of minimizing $s+t$ subject to $(s,t) \in \mathcal{L}_{n,d}$, and minimizing $t$ subject to $(t,t) \in \mathcal{L}_{n,d}$ respectively.  
\begin{proof}[Proof of Proposition \ref{prop: equivalences}]
By Lemma \ref{lem: s,t <= 0} we may assume that $\max \{s,t\} \leq 0$.
Observe that both problems optimize a continuous function over a compact set. Thus, the optimal values are attained.

Let $M \in \Sym_n$ with $M \sim \diag (1,\ldots,1,s,t)$. This implies that 
\[
I-M \sim \diag(0,\ldots,0,1-s,1-t) \succeq 0,
\]
so there exists a matrix $B \in \RR^{2 \times n}$ with $I-M=B^\top B$. Then $M \in \S^{n,d}$ if and only if $I_d-B_S^\top B_S \succeq 0$ for all $S \in \binom{[n]}{d}$. Note that this is equivalent to $\max_{S \in \binom{[n]}{d}} \rho(B_S^\top B_S) \leq 1$.
Let $A \in \RR^{2 \times n}$ be proportional to $B$ and express its columns in polar coordinates: 
\begin{align}\label{eq: matrix A}
A=\begin{bmatrix}
a_1\cos\b_1 & a_2\cos\b_2 & \dots & a_n\cos\b_n\\
a_1\sin\b_1 & a_2\sin\b_2 & \dots & a_n\sin\b_n
\end{bmatrix}.
\end{align}
For every index set $S \in \binom{[n]}{d}$, the two largest eigenvalues of $A_S^\top A_S$ coincide with the two eigenvalues of $A_SA_S^\top$. 
To compute the eigenvalues of $A_SA_S^\top$, we calculate its characteristic polynomial $x^2-\tr(A_SA_S^\top)x+\det(A_SA_S^\top)$, where $\tr(A_SA_S^\top)=\sum_{j \in S}a_j^2$. The determinant can be calculated by the Cauchy-Binet formula:
\begin{align*}
&\det(A_SA_S^\top) = \sum_{1 \leq k < l \leq d} \left(\det\begin{bmatrix}
    a_{j_k} \cos\beta_{j_k} & a_{j_l} \cos\beta_{j_l} \\ a_{j_k} \sin\beta_{j_k} & a_{j_l} \sin\beta_{j_l}
\end{bmatrix} \right)^2 \\ 
&= \sum_{1 \leq k < l \leq d}\left( a_{j_k}a_{j_l}\cos\beta_{j_k}\sin\beta_{j_l}-a_{j_k}a_{j_l}\cos\beta_{j_l}\sin\beta_{j_k}\right)^2 =\sum_{1 \leq k < l \leq d} a_{j_k}^2a_{j_l}^2\sin^2(\beta_{j_l}-\beta_{j_k})~.    
\end{align*} 
Setting $r_k:=a_k^2, \theta_k:=2\b_k$, we calculate with the double-angle identities the largest eigenvalue $\lambda$:
\begin{align*}
    2\lambda&=\sum_{j \in S} a_j^2+\sqrt{\left(\sum_{j \in S} a_j^2\right)^2-4\sum_{k<l \in S}a_k^2a_l^2\sin^2(\b_k-\b_l)}\\
    &=\sum_{j \in S} a_j^2+\sqrt{\left(\sum_{j \in S} a_j^2\right)^2-2\sum_{k<l \in S}a_k^2a_l^2(1-\cos(\theta_k-\theta_l))}\\
    &=\sum_{j \in S} a_j^2+\sqrt{\sum_{j \in S} a_j^4+2\sum_{k<l \in S}a_k^2a_l^2\cos(\theta_k-\theta_l)}\\
    &=\sum_{j \in S} r_j+\sqrt{\sum_{j \in S} r_j^2+2\sum_{k<l \in S}r_kr_l\cos(\theta_k-\theta_l)}\\
    &=\sum_{j\in S}|z_j|+\left|\sum_{j\in S} z_j\right|,
\end{align*}
where $z_k:=r_k(\cos\theta_k+i\sin \theta_k)$. Thus, the maximum eigenvalue $\alpha := \max_{S \in \binom{[n]}{d}}\rho(A_S^\top A_S)$ is exactly $\max_{S \in \binom{[n]}{d}} \frac{1}{2}\sum_{j\in S}|z_j|+\left|\sum_{j\in S} z_j\right|$.
We now proceed with proving each relation separately.\\
(1) We want to verify the relation between $\mu_{n,d}^{\tr}$ and $\mu_{n,d}^{\max}$. Since we optimize a linear function over $\mathcal{L}_{n,d}$, we may assume that the optimal value $(s,t)$ is attained at the boundary of $\mathcal{L}_{n,d}$. In particular, the corresponding matrix $M \in \S^{n,d}$ must be on the boundary of $\S^{n,d}$. Thus, $\max_{S \in \binom{[n]}{d}} \rho(B_S^\top B_S)=1$ holds. We rescale $B$ by its Frobenius norm and consider $A := \frac{1}{\|B\|_F} B$. Then, we have $\|A\|_F^2 =1$ and $ 1 = \max_{S \in \binom{[n]}{d}} \rho(B_S^\top B_S) = \|B\|_F^2 \alpha.$ Moreover,
    \begin{align}\label{eq: 2-s-t = 1/alpha}
    2-(s + t) = \tr (I-M) = \tr (B^\top B)  = \|B\|_{F}^2 = \frac{1}{\alpha},
    \end{align}
    so minimizing $s+t$ is equivalent to minimizing $\alpha$. Since 
    \[
    \|A\|_F^2 = \tr(AA^\top) = \sum_{k=1}^n a_k^2 = \sum_{k=1}^n |z_k|=1,
    \]
    we observe that our variables exactly parametrize the feasible region of the optimization problem $\mu_{n,d}^{\max}$, which proves $\alpha = \frac{1}{2} \mu_{n,d}^{\max}$ and together with \eqref{eq: 2-s-t = 1/alpha} we obtain \[\mu_{n,d}^{\tr}=2\left(1-\frac{1}{\mu_{n,d}^{\max}}\right).\]
    Let $z_1,\dots,z_n$ be as above and let $(s,t) \in \mathcal{L}_d$ be a corresponding optimal solution attaining $\mu_{n,d}^{\tr}$. Then $\diag(0,\dots,0,1-s,1-t)\sim\frac{2}{\mu_{n,d}^{\max}}A^\top A$. The two largest eigenvalues of $A^\top A$ are the same as the eigenvalues of $AA^\top$ which are $\frac{1}{2}(1\pm | \sum_{j=1}^nz_j |),$
since $\sum_{j=1}^n |z_j| =1$.
Therefore, $1-s=\frac{1}{\mu_{n,d}^{\max}}(1-|\sum_{j=1}^n z_j|)$ and $1-t=\frac{1}{\mu_{n,d}^{\max}}(1+|\sum_{j=1}^n z_j|)$, up to permutation, which implies the claim.

(2) Next, we assume that $M \sim \diag(1,\ldots,1,t,t)$. The spectral decomposition of $M$ is $M=I+(t-1)P$, where $P$ is the orthogonal projection matrix onto the $t$-eigenspace. Since $P$ has rank 2, we can factor $P=A^\top A$ for some matrix $A \in \RR^{2 \times n}$ with orthonormal rows. Then $M \in \S^{n,d}$ if and only if  $1+(t-1)\alpha \geq 0$, where $\alpha = \max_{S \in \binom{[n]}{d}}\rho(A_SA_S^\top)$. Observe that $\alpha > 0$ since at least one diagonal entry of $P$ is positive. Thus, $M \in \S^{n,d}$ if and only if $t \geq 1- \frac{1}{\alpha}$. So, minimizing $t$ is equivalent to minimizing $\alpha$. 
    As in \eqref{eq: matrix A}, we write the columns of $A$ in polar coordinates. Because $A$ has orthonormal rows, we get $AA^\top = I_2$ which is equivalent to the three conditions: $\sum_{k=1}^n a_k^2 \cos^2 \beta_k = 1, \sum_{k=1}^n a_k^2 \sin^2 \beta_k =1$, and $\sum_{k=1}^n a_k^2 \cos\beta_k \sin \beta_k = 0$. We also have $z_k = a_k^2(\cos (2\beta_k) + \mathrm{i} \sin (2\beta_k))$, so 
    the three conditions are equivalent to $\sum_{k=1}^n |z_k| =2$ and $\sum_{k=1}^n z_k = 0$. This yields
    \[ \alpha = \frac{1}{2}\max_{S\in\binom{[n]}d} \sum_{s\in S}|z_s|+\left|\sum_{s\in S} z_s\right|. \] To align this with the feasible region of the optimization problem of $\pi_{n,d}^{\max}$ we rescale to complex numbers $\sum_{k=1}^n |z_k|=1$ and this proves the relation
    \[ \pi_{n,d}^{\tr} = 1- \frac{1}{\pi_{n,d}^{\max}}.\]

\end{proof}

\section{Proofs of Theorem \ref{theorem: minmax d=2} and Proposition \ref{prop: two of length one}}\label{sec:appendixC}
We present a proof of Theorem \ref{theorem: minmax d=2} which establishes the inequality
    \[
        \max_{1\le k<l\le n}|z_k|+|z_l|+|z_k+z_l|\ge\frac4{n+1}\sum_{k=1}^n|z_k|
    \]
    for all $n \geq 3$ which settles Conjecture \ref{conjecture: optimal configurations} for $d=2$. We also present a proof of Proposition \ref{prop: two of length one}.

We refer to the sketch of the proof of Theorem \ref{theorem: minmax d=2} and continue from there. By positive rescaling one can set $C=1$, which we will assume from now on. In this case $S=\sum_{k=1}^n r_k=\frac{n+1}4$, $r_k\in[\frac14,\frac12]$ for $k=1,\dots,n$, and 
\[
F(r_1\dots,r_n)=\sum_{k=1}^n\arccos\left(\frac{1}{2r_kr_{k+1}}-\frac{1}{r_k}-\frac{1}{r_{k+1}}+1\right)~.
\]
Our goal is to show that, with these restrictions, $F(r_1,\dots,r_n)\ge2\pi$. 

Let $x_k=4r_k-1$. Then we have $x_k\in [0,1]$ and $\sum_{k=1}^n x_k=1$. Define $$\psi(x):=\sqrt{\frac{1-x}{1+x}}.$$
Observe that $\psi(x_k)=\sqrt{\frac{1}{2r_k}-1}$.
Then we have
\[\frac{1}{2r_kr_{k+1}}-\frac{1}{r_k}-\frac{1}{r_{k+1}}+1=2\left(\frac{1}{2r_k}-1\right)\left(\frac{1}{2r_{k+1}}-1 \right)-1=2\psi(x_k)^2\psi(x_{k+1})^2-1.\]
From the doubling angle cosine identity we get \[\arccos \left(2\psi(x_k)^2\psi(x_{k+1})^2-1\right)=2\arccos \left(\psi(x_k)\psi(x_{k+1})\right)  .\]
Therefore we may state our problem as follows:

\begin{center}
    If $(x_1,\dots,x_n)$ lies in the probability simplex $\Delta^n$, then $\sum_{k=1}^n\arccos(\psi(x_k)\psi(x_{k+1}))\geq\pi$.
\end{center}

Recall that $\arccos(-t)=\pi-\arccos t$ for all $t\in[-1,1]$. 
\begin{lemma}\label{lemma: identity}
    For any $x,y,z\in[0,1]$ satisfying $x^2+y^2+z^2+2xyz=1$, it holds that
    \[
    \arccos x+\arccos y+\arccos z=\pi~.
    \]
\end{lemma}
\begin{proof}
    Since $z\in[0,1]$, the quadratic equation $z^2+(2xy)z+(x^2+y^2-1)=0$ yields 
    \[
    z=\frac{-2xy+\sqrt{4x^2y^2-4(x^2+y^2-1)}}2=-xy+\sqrt{(1-x^2)(1-y^2)}
    \]
    so if $x=\cos\alpha$ and $y=\cos\beta$, with $\alpha,\beta\in[0,\frac{\pi}2]$, then $z=-\cos(\alpha+\beta)$. Hence,
    \[
        \arccos x+\arccos y+\arccos z=\alpha+\beta+\arccos(-\cos(\alpha+\beta))=\alpha+\beta+(\pi-(\alpha+\beta))=\pi~.
    \]
\end{proof}
\begin{corollary}\label{cor: arcos n=3}
    If $x_1,x_2,x_3\in[0,1]$ satisfy $x_1+x_2+x_3=1$, then $\sum_{k=1}^3\arccos(\psi(x_k)\psi(x_{k+1}))=\pi~$.
\end{corollary}
\begin{proof}
Let $x:=\psi(x_1)\psi(x_2)$, $y:=\psi(x_2)\psi(x_3)$, $z:=\psi(x_3)\psi(x_1)$. It is routine to verify that, given $x_1+x_2+x_3=1$, $x,y,z$ satisfy the Lemma \ref{lemma: identity} identity, i.e.,
\begin{align*}
    \frac{(1-x_1)(1-x_2)}{(1+x_1)(1+x_2)}+\frac{(1-x_2)(1-x_3)}{(1+x_2)(1+x_3)}+\frac{(1-x_3)(1-x_1)}{(1+x_3)(1+x_1)}+2\frac{(1-x_1)(1-x_2)(1-x_3)}{(1+x_1)(1+x_2)(1+x_3)}=1~.
\end{align*}
\end{proof}

\begin{lemma}\label{lemma: arccos psi}
    For any $x,y\in[0,1]$ with $x+y\le1$,
    \[
    \arccos(\psi(x))+\arccos(\psi(y))+\arccos(\psi(x)\psi(y))\ge2\arccos(\psi(x+y))~.
    \]
    Equality holds if and only if $x=0$, or $y=0$, or $x+y=1$.
\end{lemma}
\begin{proof}
    If $x=0$ or $y=0$, the inequality trivially becomes an equality. If $x+y=1$, observe that, as in the proof of Corollary \ref{cor: arcos n=3}, $\psi(x)$, $\psi(1-x)$, $\psi(x)\psi(1-x)$ satisfy the identity of Lemma \ref{lemma: identity}, so there is also equality in this case.
    
    Suppose then that $x,y>0$ and $x+y<1$. Let $\cos\alpha=\psi(x)$, $\cos\beta=\psi(y)$, $\cos\gamma=\psi(x+y)$, with $\alpha,\beta,\gamma\in(0,\frac{\pi}2)$. We need to show that
    \[
    \alpha+\beta+\arccos(\cos\alpha\cos\beta)>2\gamma~,
    \]
    or equivalently,
    \begin{align*}
        \arccos(\cos\alpha\cos\beta)&> 2\gamma-(\alpha+\beta)\\
        \Leftrightarrow \cos\alpha\cos\beta&< \cos(2\gamma-(\alpha+\beta))\\
        \Leftrightarrow \cos\alpha\cos\beta&< \cos2\gamma\cos(\alpha+\beta)+\sin2\gamma\sin(\alpha+\beta)\\
        \Leftrightarrow \cos\alpha\cos\beta&< \cos2\gamma(\cos\alpha\cos\beta-\sin\alpha\sin\beta)+\sin2\gamma(\sin\alpha\cos\beta+\sin\beta\cos\alpha)\\
        \Leftrightarrow 1&< \cos2\gamma(1-\tan\alpha\tan\beta)+\sin2\gamma(\tan\alpha+\tan\beta)\\
        \Leftrightarrow 1&< (1-pq)\cos2\gamma+(p+q)\sin2\gamma
    \end{align*}
    where $p:=\tan\alpha$, and $q:=\tan\beta$.

    Observe that $\cos\alpha=\sqrt{\frac{1-x}{1+x}}$ and so $\sin\alpha=\sqrt{\frac{2x}{1+x}}$, therefore $p=\tan\alpha=\sqrt{\frac{2x}{1-x}}$, from where we obtain $x=\frac{p^2}{2+p^2}$. Similarly, $y=\frac{q^2}{2+q^2}$, and therefore
    \begin{align*}
        z:=x+y=\frac{p^2}{2+p^2}+\frac{q^2}{2+q^2}=\frac{2p^2+2q^2+2p^2q^2}{4+2p^2+2q^2+p^2q^2}=\frac{2A+2B^2}{4+2A+B^2}~
    \end{align*}
    where $A:=p^2+q^2$ and $B:=pq$. Observe that $A\ge 2B>0$ since $p^2+q^2-2pq=(p-q)^2$ and $x,y\in(0,1)$, and also that $B<2$ because $x+y<1$ implies $pq=2\sqrt{\frac{xy}{(1-x)(1-y)}}<2\sqrt{\frac{xy}{xy}}=2$.

    Also observe that $\cos\gamma=\sqrt{\frac{1-z}{1+z}}$, $\sin\gamma=\sqrt{\frac{2z}{1+z}}$, and so $\cos2\gamma=2\cos^2\gamma-1=\frac{1-3z}{1+z}$, and $\sin2\gamma=2\sin\gamma\cos\gamma=\frac{2\sqrt{2z(1-z)}}{1+z}$. 

    We thus need to show that
    \begin{align*}
        1 &< (1-pq)\left(\frac{1-3z}{1+z}\right)+(p+q)\left(\frac{2\sqrt{2z(1-z)}}{1+z}\right)\\
       \Leftrightarrow 1+z &< (1-pq)(1-3z)+2(p+q)\sqrt{2z(1-z)}\\
       \Leftrightarrow 4z &< 3pqz-pq+2(p+q)\sqrt{2z(1-z)}\\
       \Leftrightarrow z(4-3pq)+pq &< 2(p+q)\sqrt{2z(1-z)}
    \end{align*}
    or equivalently, changing to $A$ and $B$, $z(4-3B)+B < 2\sqrt{A+2B}\sqrt{2z(1-z)}\Leftrightarrow$
    \begin{align*}
        \left(\frac{2A+2B^2}{4+2A+B^2}\right)(4-3B)+B < 2\sqrt{A+2B}\sqrt{2\left(\frac{2A+2B^2}{4+2A+B^2}\right)\left(\frac{4-B^2}{4+2A+B^2}\right)}\\
        \Leftrightarrow (2A+2B^2)(4-3B)+B(4+2A+B^2) < 4\sqrt{A+2B}\sqrt{(A+B^2)(4-B^2)}~.
    \end{align*}   
    With the substitution $A=2B+t$ (so $t\ge0$ since $A\ge 2B$), one can verify that the LHS of the last inequality is positive (it is equivalent to $5B(4-B^2)+4t(2-B)>0$). So, squaring both sides, and then computing the difference RHS$-$LHS using $A=2B+t$, we obtain
    \begin{align*}
    (8A+4B-4AB+8B^2-5B^3)^2 < 16(A+2B)(A+B^2)(4-B^2)\\
    \Leftrightarrow B(2-B)(25B^4+114B^3+156B^2+56B + (56B^2+128B+32)t + 32t^2)>0
    \end{align*}
    because $B\in(0,2)$ and $t\ge0$.
\end{proof}

\begin{lemma}\label{lemma: arccos sum}
    For any $x,y\in[0,1]$,
    \[
    \arccos x+\arccos y=\arccos\left(xy-\sqrt{1-x^2}\sqrt{1-y^2}\right)~.
    \]
\end{lemma}
\begin{proof}
    If $\cos\alpha=x$ and $\cos\beta=y$ with $\alpha,\beta\in[0,\frac{\pi}{2}]$ then $\sin\alpha=\sqrt{1-x^2}$, $\sin\beta=\sqrt{1-y^2}$, and
    \begin{align*}
        \cos(\alpha+\beta)&=\cos\alpha\cos\beta-\sin\alpha\sin\beta =xy-\sqrt{1-x^2}\sqrt{1-y^2}\\
        \Rightarrow \alpha+\beta&=\arccos\left(xy-\sqrt{1-x^2}\sqrt{1-y^2}\right)~.
    \end{align*}
\end{proof}

Define the function $h:[0,1]\times[0,1]\to[0,\frac{\pi}{2}]$ by
\[
h(x,y)=\arccos(\psi(x)\psi(y)).
\]

\begin{lemma}\label{lemma: arccos triangle ieq}
    For any $a,b,c,d\in[0,1]$ with $a+b+c+d\le1$,
    \[
    h(a,b)+h(b,c)+h(c,d)\ge h(a,b+c)+h(b+c,d)
    \]
    and equality is attained if and only if one of the following holds: 
    \begin{center}
    (i) $a=b=0$, or (ii) $b=c=0$, or (iii) $c=d=0$, or (iv) $a=d=0$ and $b+c=1$.
    \end{center}
\end{lemma}
\begin{proof}
Observe that if $b=0$ then the inequality to prove becomes 
\[
\arccos(\psi(a))+\arccos(\psi(c))\ge\arccos(\psi(a)\psi(c)),
\] 
which follows from the arccosine sum identity and the fact that arccosine is decreasing. Namely, $\psi(a)\psi(c)-\sqrt{1-\psi(a)^2}\sqrt{1-\psi(c)^2}\le \psi(a)\psi(c)$. Observe, since arccosine is strictly decreasing, that equality occurs if and only if $a=0$ or $c=0$, which correspond to cases (i), (ii). 

Similarly, if initially $c=0$, then the inequality holds and equality is attained if and only if $b=0$ or $d=0$, corresponding to cases (ii), (iii). 

From now on assume $b,c>0$. We will show that the function 
\[
G(a,b,c,d):=h(a,b)+h(b,c)+h(c,d)-h(a,b+c)-h(b+c,d)
\]
is nonnegative. To this end we first show that $\frac{\partial G}{\partial a}>0$ and $\frac{\partial G}{\partial d}>0$ for $a,d\in[0,1)$.

Set $x:=\psi(a)$, $y:=\psi(b)$, $z:=\psi(b+c)$, and observe that $y>z\ge0$ because $c>0$ and $\psi$ is strictly decreasing. Further, observe that $\psi'(t)=\frac{-1}{\psi(t)(1+t)^2}<0$ on $[0,1)$. We have
\begin{align*}
\frac{\partial G}{\partial a}=h_a(a,b)-h_a(a,b+c)=\frac{-\psi'(a)y}{\sqrt{1-x^2y^2}}+\frac{\psi'(a)z}{\sqrt{1-x^2z^2}}>0\\
\Leftrightarrow \frac{-\psi'(a)y}{\sqrt{1-x^2y^2}}>\frac{-\psi'(a)z}{\sqrt{1-x^2z^2}}
\Leftrightarrow \frac{y}{\sqrt{1-x^2y^2}}>\frac{z}{\sqrt{1-x^2z^2}}\\
\Leftrightarrow y^2(1-x^2z^2)>z^2(1-x^2y^2)
\Leftrightarrow y^2>z^2
\end{align*}
which is trivial. By symmetry, $\frac{\partial G}{\partial d}>0$ also.

Since we restrict $G$ to the domain $\Delta:=\{(a,b,c,d)\,:\,a+b+c+d\le1,\,\,\,\, a,b,c,d\in[0,1]\}$, both partial derivatives $\frac{\partial G}{\partial a}$ and $\frac{\partial G}{\partial d}$ being positive in its interior, the continuity of $G$ on $\Delta$ allows us to conclude that
\begin{align}\label{eqtn: G}
G(a,b,c,d)\ge G(0,b,c,d)\ge G(0,b,c,0)
\end{align}
for any point $(a,b,c,d)\in\Delta$. This is because the straight line segment from $(a,b,c,d)$ to $(0,b,c,d)$, and the straight line segment from $(0,b,c,d)$ to $(0,b,c,0)$ are contained in $\Delta$.

So it is enough to show that $G(0,b,c,0)\ge0$, i.e.
\[
\arccos(\psi(b))+\arccos(\psi(b)\psi(c))+\arccos(\psi(c))\ge2\arccos(\psi(b+c))~,
\]
and this follows by Lemma \ref{lemma: arccos psi}.

Finally, for equality to hold in this case, it is necessary that $a=d=0$. If not, then at least one inequality in \eqref{eqtn: G} is strict. Then, by Lemma \ref{lemma: arccos psi}, equality occurs if and only if $b+c=1$. This is precisely case (iv), and we see that equality indeed holds. 
\end{proof}

\begin{proposition}\label{prop: spherical inequality}
    For any $x_1,\dots,x_n\in[0,1]$ with $x_1+\dots+x_n=1$, $n\ge3$,
    \[
    \sum_{k=1}^n h(x_k,x_{k+1})\ge\pi~.
    \]
\end{proposition}
\begin{proof}
    For any positive integer $m\ge3$, and real numbers $y_1,\dots,y_m\in[0,1]$, let 
    \[
    H(y_1,\dots,y_m):=\sum_{k=1}^m h(y_k,y_{k+1})~.
    \]
    
    We induct on $n\ge3$. The base case is Corollary \ref{cor: arcos n=3} and the inductive step uses Lemma \ref{lemma: arccos triangle ieq}. By the induction hypothesis, we have $H(x_1,\dots,x_{n-2},x_{n-1}+x_n)\ge\pi$.
    
    We now show that $H(x_1,\dots,x_n)\ge H(x_1,\dots,x_{n-2},x_{n-1}+x_n)$, i.e.
    \begin{align*}
        \sum_{k=1}^n h(x_k,x_{k+1})\ge \left(\sum_{k=1}^{n-3}h(x_k,x_{k+1})\right)+h(x_{n-2},x_{n-1}+x_n)+h(x_{n-1}+x_n,x_1)\\
        \Leftrightarrow h(x_{n-2},x_{n-1})+h(x_{n-1},x_n)+h(x_n,x_1)\ge h(x_{n-2},x_{n-1}+x_n)+h(x_{n-1}+x_n,x_1)
    \end{align*}
    and this follows by Lemma \ref{lemma: arccos triangle ieq}.
\end{proof}

The proof of Theorem \ref{theorem: minmax d=2} now follows directly from the sketch of its proof in Section \ref{sec:Conjectured optimal value}, the discussion at the beginning of Appendix \ref{sec:appendixC}, and Proposition \ref{prop: spherical inequality}.

We now prove Proposition \ref{prop: two of length one} which served as a stepping stone for the classification of the minimizers for the min-max optimal point configuration problem when $d=2$. 

\begin{proof}[Proof of Proposition \ref{prop: two of length one}]
    Recall that $n+1=S=\sum_{k=1}^n r_k$ and $z_1,\ldots,z_n$ are ordered increasingly by argument. By the proof of Theorem \ref{theorem: minmax d=2}, optimal configurations must satisfy $1\le r_k\le 2$ for $k=1,\dots,n$.
    Let $x_k = r_k-1$. Since $r_k \in [1,2]$ we have $x_k \in [0,1]$ and $\sum_{k=1}^n x_k = 1$. So $(x_1,\dots,x_n)$ is in the probability simplex.
    Moreover, $|z_k|+|z_{k+1}|+|z_k+z_{k+1}|=4$ for $k=1,\dots,n$, otherwise there is at least one strict inequality $|z_k|+|z_{k+1}|+|z_k+z_{k+1}|<4$ leading to the same contradiction of the proof of Theorem \ref{theorem: minmax d=2}. 
    Furthermore, the sum of the angles between all pairs of consecutive vectors must be exactly $2\pi$, i.e., $\sum_{k=1}^n \gamma_k = 2\pi$.
    Since $F(r_1,\ldots,r_n)$ satisfies the chain of inequalities $\sum_{k=1}^n \gamma_k\ge F(r_1,\ldots,z_n)\ge2\pi$, all these inequalities are in fact equalities.
    
    The function $F$ corresponds to twice the cyclic sum $\sum_{k=1}^n h(x_k, x_{k + 1})$, where 
    \[
    h : [0,1] \times [0,1] \to [0,\frac{\pi}{2}], \quad h(x,y)=\arccos(\psi(x)\psi(y))\]
    was introduced before Lemma \ref{lemma: arccos triangle ieq}. By Lemma \ref{lemma: arccos triangle ieq} we have the inequality: 
    \[ h(x_1,x_2)+h(x_2,x_3)+h(x_3,x_4)+h(x_4,x_1) \geq h(x_1,x_2) + h(x_2,x_3+x_4)+h(x_3+x_4,x_1).\]
    Applying this repeatedly, we get
    \[
    \sum_{k=1}^n h(x_k, x_{k+1}) \geq h(x_1, x_2) + h(x_2, x_3 + \cdots + x_n) + h(x_3 + \cdots +x_n, x_1).
    \]
    By Corollary \ref{cor: arcos n=3} the right hand side of the inequality is exactly $\pi$, since the arguments $x_1$, $x_2$, and $x_3+\cdots+x_n$ sum to $1$. Thus, $F(r_1,\dots,r_n) \geq 2\pi$ and since we already established that $F(r_1,\dots,r_n) = 2\pi$ holds, the equality condition in Lemma \ref{lemma: arccos triangle ieq} applies. 
    Thus, at least two adjacent variables are zero. Therefore, there exists an index $k$ such that $x_k=0$ and $x_{k+1}=0$. This implies $r_k=r_{k+1}=1$. 
    For such $k$ we have $|z_k|+|z_{k+1}|+|z_k+z_{k+1}|=4$, so $|z_k+z_{k+1}|=2=|z_k|+|z_{k+1}|$ and therefore $z_k=z_{k+1}$.
    By the triangle inequality we obtain that the two vectors of length one are equal. Since the sum of all lengths is $S=n+1$ we find $|z_k|=|z_{k+1}|=\frac{1}{n+1}(|z_1|+\dots+|z_n|)$.
\end{proof}

\subsection*{Declaration of generative AI and AI-assisted technologies in the manuscript preparation process}

The authors acknowledge AI assistance (Aristotle, Claude, DeepSeek, Gemini) in formulating parts of the proofs in Section \ref{sec:nonconv}, of Proposition \ref{proposition: 2-local Fischer} and of Theorem \ref{theorem: minmax d=2}. The final versions of the proofs were written entirely by the authors. The authors assume full responsibility for all content.

\printbibliography
\end{document}